\documentclass[11pt]{article}
\usepackage{xcolor}
\usepackage{amsfonts}
\usepackage{amssymb}
\usepackage{epsfig}
\usepackage{CJK}
 \usepackage{setspace}

\usepackage{amsthm}
\usepackage{amsmath,bm}
\usepackage{graphicx}
\usepackage[citecolor=red]{hyperref}
\usepackage{cite}

 \newtheorem{thm}{Theorem}[section]
 
 \newtheorem{lem}[thm]{Lemma}
 \newtheorem{prop}[thm]{Proposition}
 \theoremstyle{definition}
 
 \newtheorem{rem}[thm]{Remark}
 \numberwithin{equation}{section}
 \def\bR{\mathbb{R}}

\def\inte#1{
\displaystyle\mathop{#1\kern0pt}^\circ }

\let\p=\partial

\let\f=\frac

\let\p=\psi

\let\D=\Delta

\def\pa{\partial}

\def\virgp{\raise 2pt\hbox{,}}
\def\cdotpv{\raise 2pt\hbox{;}}

\def\C{\mathop{\mathbb C\kern 0pt}\nolimits}
\def\DD{\mathop{\mathbb D\kern 0pt}\nolimits}
\def\EE{\mathop{{\mathbb E \kern 0pt}}\nolimits}
\def\K{\mathop{\mathbb K\kern 0pt}\nolimits}
\def\N{\mathop{\mathbb N\kern 0pt}\nolimits}
\def\Q{\mathop{\mathbb Q\kern 0pt}\nolimits}
\def\R{\mathop{\mathbb R\kern 0pt}\nolimits}
\def\SS{\mathop{\mathbb S\kern 0pt}\nolimits}
\def\ZZ{\mathop{\mathbb Z\kern 0pt}\nolimits}
\def\TT{\mathop{\mathbb T\kern 0pt}\nolimits}
\def\P{\mathop{\mathbb P\kern 0pt}\nolimits}

\def\dv{\mbox{div}}

\def\na{\nabla}
\def\p{\partial}

\newtheorem{theorem}{Theorem}[section]

\newtheorem{Lemma A.1}{Lemma A.1}
\theoremstyle{definition}
\newtheorem{definition}[theorem]{Definition}

\theoremstyle{remark}

\newcommand{\Real}{\mathbb R}
\newcommand{\eps}{\epsilon}

\newcommand{\beq}{\begin{equation}}
\newcommand{\eeq}{\end{equation}}
\newcommand{\ben}{\begin{eqnarray}}
\newcommand{\een}{\end{eqnarray}}
\newcommand{\beno}{\begin{eqnarray*}}
\newcommand{\eeno}{\end{eqnarray*}}

\begin{document}
\title{On the energy method for the global solutions to the three dimensional incompressible non-resistive MHD  near equilibrium}

\author{Yuan Cai\footnote{School of Mathematical Sciences, Fudan University, Shanghai, 200433, P. R. China.  Email: caiy@fudan.edu.cn}
\and
Bin Han\footnote{School of Mathematical Sciences, Hangzhou Dianzi  University,
Hangzhou, 310018, P. R. China. Email: hanbin@hdu.edu.cn.}
\and
Na Zhao\footnote{School of Mathematics, Shanghai  University of Finance and Economics, Shanghai, 200433, P. R. China. Email: zhaona@shufe.edu.cn.}
}

\date{}
\maketitle
\begin{abstract}
We prove the global existence of the smooth solutions near equilibrium to
the Cauchy problem of the incompressible non-resistive magnetohydrodynamic equations in the whole three  dimensional space under some admissible condition. The result has been obtained by Xu and Zhang (SIAM J. Math. Anal. 47: 26--65, 2015) in anisotropic Besov space framework. In this paper, we provide a new proof based on the temporal weighted energy method.
\end{abstract}

{\bf Key words.}  MHD, Lagrangian coordinates, global well-posedness, energy estimates\\

{\bf AMS subject classifications.} 35Q30, 76D03, 76N10, 76W05

\section{Introduction}
The magnetohydrodynamic systems are fundamental equations in magnetohydrodynamics (MHD)
where the study of this field was initiated by Hannes Alfv\'en \cite{Al, Al2} who won Nobel Prize in 1970.  
They reflect the basic physical laws governing the motion of electrically conducting fluids, such as plasma, liquid metals and electrolytes. The MHD equations share similarities with the Navier-Stokes equations, but they contain richer mathematical structure.
In this article, we consider the global existence of strong solutions to the following three dimensional incompressible viscous and non-resistive magnetohydrodynamic system
\begin{equation}\label{A1}
\begin{cases}
\p_t u+u\cdot \na u-\Delta u+\nabla p=b\cdot \nabla b,\\
\p_t b+u\cdot\nabla b=b\cdot\nabla u, \\
\hbox{div}\, u=\hbox{div}\, b=0, \\
(b,u)|_{t=0}=(b_0,u_0),
\end{cases}
\end{equation}
where $b=(b_1,b_2,b_3)^{\top}, u=(u_1,u_2,u_3)^{\top}$  represent the magnetic field and velocity field respectively, $p$ is the scalar pressure.
The velocity field obeys the Navier-Stokes equations with Lorentz force. 
 The magnetic field satisfies  the non-resistive Maxwell-Faraday equations which describe the Faraday's law of induction. The viscous non-resistive MHD system is not merely a combination of the Navier-Stokes and the transport equations but an interactive and integrated system. Mathematically, due to the losing resistivity of the magnetic equations, it is difficult to control the Lorentz force in the momentum equations.
One may check the references \cite{Bis,Davi,PF} for detailed explanations to this system.

Let us briefly  recall some well-known results to MHD systems.
Firstly, in the case of $\Real^d$, for  the viscous and resistive homogeneous MHD system, Duvaut and Lions \cite{DL} established the global
existence and uniqueness of the solution in classical Sobolev spaces
for small initial data.
The local well-posedness of classical solutions for fully viscous MHD systems
was established by Sermange and Temam \cite{ST}, in which the global well-posedness was also proved in two
dimensions.

For the viscous and non-resistive problem, Lin and Zhang \cite{LZ14} established the global solutions for a
three dimensional  MHD model with initial data sufficiently close to the equilibrium state.
We also refer to  Lin and Zhang \cite{LZ} for an elementary proof.
For the physical  system \eqref{A1} in two dimensional case,  Lin, Xu and Zhang \cite {LXZ1} constructed  the global  smooth solutions  around the equilibrium
by imposing  some admissible conditions.
Later on,  the global existence of small solutions without imposing such admissible conditions on the initial magnetic field was obtained by Ren, Wu, Xiang and Zhang  \cite{RWXZ} (see \cite{ZhangT} for a simplified proof).
 For system \eqref{A1} in three dimensional case, the global well-posedness result  was obtained by Xu and Zhang  \cite{XZ}  by introducing the Lagrangian reformulation of the problem, and by imposing  some admissible conditions to the initial magnetic field 
 as in \cite{LXZ1}.  Such admissible conditions were removed  in \cite{AZ} by Abidi and Zhang
under a more intrinsic Lagrangian reformulation.
 The  existence of global solutions  in periodic domain was obtained by Pan, Zhou and Zhu \cite{PZZ}.
 On the other hand,
 Zhang  \cite{ZhangT2} considered the two dimensional case where the background magnetic field is $(\epsilon^{-1},0)$.  Zhai and  Zhang \cite{CuiZ} studied the stability problem when the solution is sufficiently close to a special solution with linearly growing velocity.
 With some odevity conditions, Jiang and Jiang \cite{JiangJ} proved the existence and uniqueness of strong solutions with some large initial perturbations in two dimensional periodic domains under Lagrangian coordinates. In addition, under the axially symmetric setting,  Lei \cite{Lei} proved that  the $H^2$ initial data can generate a unique global large solution of MHD system \eqref{A1}. Lei and  Zhou \cite{LeiZ} constructed the global weak solutions for the two dimensional incompressible resistive MHD system.
 For the local in time existence of low regularity solutions to the three dimensional incompressible non-resistive MHD,
 Chemin, McCormick, Robinson and Rodrigo \cite{CMRR} proved the sharp local well-posedness in Besov spaces.
Fefferman, McCormick, Robinson and Rodrigo \cite{Fe1, Fe2} obtained the local in time existence result in nearly optimal Sobolev spaces.




For the two dimensional  viscous and non-resistive compressible MHD system,  the global existence result of classical solutions was established by Wu, Wu \cite{WW} in whole space and by Wu,  Zhu \cite{WZ-2} in periodic domain.  In the three dimensional case,  Hu and Wang \cite{HW} studied the existence and large time behavior of global weak solutions in a bounded domain with large data.  Tan and Wang \cite{TW} considered the  global well-posedness of the non-resistive MHD system in a flat domain $\R^2\times (0,1)$ with  vertical background  magnetic field.

For the ideal conducting fluid,  Bardos, Sulem and Sulem \cite{BSS} proved the existence of global solutions with small initial data to the MHD equations  which subject to a strong magnetic field. The global in time vanishing viscosity limit of the full diffusive MHD system to the ideal equations was obtained
 by  He, Xu, Yu \cite{HXY}, Cai, Lei \cite{CaiLei} and Wei, Zhang \cite{WeiZ}.

On the other hand, many efforts have been made on the  mixed partial dissipation and  partial magnetic diffusion in the two or three dimensional MHD system. In \cite{CRW}, Cao, Regmi and Wu established the global bound in Lebesgue spaces  to the two dimensional incompressible  MHD equations with horizontal dissipation and horizontal magnetic diffusion.  The mixed partial dissipation and magnetic diffusion and only magnetic diffusion cases were studied by Cao, Wu \cite{CW} and by Cao, Wu, Yuan\cite{CWY}. The global solutions to the two dimensional incompressible  MHD equations with only magnetic diffusion in periodic domain were proved by Zhou and Zhu \cite{ZZ}.
 Recently,  Wu and Zhu \cite{WZ} constructed the global solutions in the whole space   with horizontal dissipation and  vertical magnetic diffusion near equilibrium. For more studies on MHD, we refer to \cite{AP,CZZ,DZ,DLW,DLW18,HanZhao,HanZhao2,Hu,HL,HLi,JZ,JNW,RWXZ2,WeiZ1,WeiZ2,Wu18,WWX,YZ,Ye} and the references therein.

Before we state our main result, we recall the admissible condition by Lin, Xu, Zhang \cite{LXZ1} and Xu, Zhang \cite{XZ}.
\begin{definition}\label{def1}
	Let $b_0=(b_0^1,b_0^2,b_0^3)$ be a smooth enough vector field. We define its trajectory $\widetilde{X}(t,y)$ by
	\begin{align*}
		\begin{cases}
				\frac{\mathrm{d}}{\mathrm{d}t}\widetilde{X}(t,y)=b_0(\widetilde{X}(t,y)),\\
			\widetilde{X}(0,y)=y.
		\end{cases}
	\end{align*}
	We call that $f$ and $b_0$ are admissible on a domain $D$ of $\mathbb{R}^3$ if there holds
	\begin{align*}
		\int_\mathbb{R}f(\widetilde{X}(t, y))\mathrm{d}t=0
	\end{align*}
	for all $y\in D$.
\end{definition}
The main result of this paper is stated as follows.
\begin{thm}\label{MTH}
Let $e_1=(1,0,0)^\top$,
 $u_0\in H^3(\Real^3)$, $b_0-e_1\in H^3(\Real^3)$ with $\mathrm{div}\, u_0=\mathrm{div}\, b_0=0$. Assume that $b_0-e_1$ and $b_0$ are admissible
 on $\{0\}\times\mathbb{R}^2$ in the sense of Definition \ref{def1} and
 $\mathrm{supp} (b_0-e_1)
 (\cdot, x_2, x_3)\subset [-K, K]$ for
 some positive constant $K$. Then
 there exists a sufficiently small positive constant
$\eps_0$  such that if
\begin{align*}
\|u_0\|_{H^3}+\|b_0-e_1\|_{H^3}\leq  \eps_0,
\end{align*}
\eqref{A1} has a unique global solution $(u,b)$ such that for any $T>0$,
$$u\in C([0,T];H^3(\Real^3)),\quad \nabla u\in L^2(0,T; H^3(\Real^3)), \quad b-e_1  \in C([0,T];H^3(\Real^3)).$$
\end{thm}
\begin{rem}
The theorem can be regarded as a new proof of the
global existence of the smooth solutions near equilibrium to
the Cauchy problem of the incompressible non-resistive magnetohydrodynamic equations in the whole three  dimensional space under the admissible condition defined in Definition \ref{def1}.
The result has been obtained by Xu and Zhang (SIAM J. Math. Anal. 47: 26--65, 2015) in anisotropic Besov space framework. We also mention that the admissible condition has been removed by  Abidi-Zhang \cite{AZ}.
\end{rem}

The proof of the main theorem will be conducted in Lagrangian coordinates.
Same to the transformation in \cite{LXZ1,XZ}, we define the flow map $X(t,y)$ by
\begin{equation*}
\begin{cases}
\f{\mathrm{d}}{\mathrm{d}t} X(t,y)=u(t,X(t,y)),\\
X(0,y)=X_0(y),
\end{cases}
\end{equation*}
where $X_0(y):\mathbb{R}^3\rightarrow \mathbb{R}^3$ is an invertible map. If $b_0$ satisfies the assumption in Theorem \ref{MTH}, then there exists a $X_0(y)$ such that
\begin{align}\label{eq-com}
	A^\top_0(y)b_0(X_0(y))=e_1,\quad \det (\nabla_y X_0)=1.
\end{align}
Here $A(t, y)$ is denoted by
	\begin{equation}\label{dfnA}
		A(t,y)=\left(\f{\p X(t,y)}{\p y}\right)^{-\top}
	\end{equation}
and $A_0(y)=A(0,y)$.
We refer to \cite{XZ}  for the details of the derivation.
The two dimensional version can be found in \cite{LXZ1}.

Then under the  Lagrangian coordinates, \eqref{A1} becomes
\begin{equation}\label{A3}
\begin{cases}
  X_{tt} - \hbox{div}_y(A^\top A\nabla_y X_t ) - \p_{y_1}^2 X + (\nabla_y X)^{- \top}\nabla_y p=0,\\
\det\, (\nabla_y X)=1,\\
X(0,y)=X_0(y),\quad X_t(0,y)=u_0(X_0(y)).
\end{cases}
\end{equation}
We refer to Section \ref{secL} for the derivation of \eqref{A3}.

This paper aims at establishing the global well-posedness for the reformulated system \eqref{A3} by using elementary energy method. The large time decay estimates are also presented.
To continue, we first define some energy and temporal weighted energy functional spaces. In the sequel, all spatial derivatives are taken with respect to the Lagrangian spatial variable $y$ without specification.
Consider the solutions near equilibrium $X(t,y)=y+Y(t,y)$.
Then the equation \eqref{A3} reduces to
\begin{equation}\label{A4}
\begin{cases}
Y_{tt} - \hbox{div}_y(A^\top A\nabla_y Y_t ) - \p_{y_{_1}}^2 Y + (I+\nabla_y Y)^{- \top}\nabla_y p=0,\\
\det(I+\nabla_y Y)=1,\\
Y(0,y)=Y_0(y),\quad Y_t(0,y)=Y_1(y),
\end{cases}
\end{equation}
where $Y_0(y)=X_0(y)-y$ and $Y_1(y)=u_0(X_0(y))$.
The energy norm $E(t)$ and the dissipative energy $D(t)$ are defined as follows
\begin{align*}
	\begin{split}
		E(t)&=\|  Y_t(t)\|_{H^2}^2
		+\| \pa_1 Y(t)\|_{H^2}^2 +\|  \D Y(t)\|_{H^2}^2+(t+1)\|  \nabla Y_t (t)\|_{H^2}^2 \\[-4mm]\\
		&\quad
		+(t+1)\|  \nabla \pa_1Y(t)\|_{H^2}^2  +(t+1)^2 \|  \nabla\p_1Y_t(t)\|_{H^1}^2
		+(t+1)^2 \|  \nabla \pa_1^2 Y (t) \|_{H^1}^2,\\[-4mm]\\
		D(t)&=
		\|\na Y_t (t)\|_{H^2}^2
		+ \|\na \pa_1 Y(t) \|_{H^2}^2+(t+1) \|\nabla\p_1^2 Y(t) \|^2_{H^1}  \\[-4mm]\\
		&\quad+(t+1) \|\D Y_t(t)\|_{H^2}^2
		+ (t+1)^2 \| \D\p_1 Y_t (t)\|_{H^1}^2.
	\end{split}
\end{align*}
We denote
\begin{equation*}
	\mathcal{E}(t)=\sup_{0\leq \tau\leq t}E(\tau)+\int_0^t  D(\tau) \,\,\mathrm{d}\tau.
\end{equation*}

Now we write the main global well-posedness result  under the Lagrangian coordinates
by the flow $Y(t,y)$ .
\begin{thm}\label{th1}
	Let $Y_0\in H^4(\Real^3)$, $Y_1 \in H^3(\Real^3)$
with $\det\, (I+\nabla_y Y_0)=1$. There exists a constant $\epsilon_0>0$ such that, if
	\begin{equation}\label{asmp}
		\|Y_1\|_{H^3}^2
		+\|\pa_1Y_0\|_{H^3}^2
		+\|  \D Y_0\|_{H^2}^2\leq \epsilon_0,
	\end{equation}
	then there exists a global
	unique solution $Y$ solving \eqref{A4}
 on $[0,\infty)$. Moreover, the solution satisfies the following  estimate
	$$	\mathcal{E}(t)\leq M\epsilon_0,$$
	where $M$ is a positive constant.
\end{thm}


%


We use the elementary energy method, anisotropic techniques and some temporal weighted norms to close the global estimates.
The key point in the proof of the main theorem is the $L^1$ in time estimate of $\| \na_y Y_t \|_{L^\infty}$. The usual energy method only yields $L^2$ in time estimate.
We apply the temporal weighted energy method to provide better integrability.
The basic strategy is that in the higher order derivative estimate, some temporal factor can be applied.
Moreover, stronger temporal weight can be applied when we have both $\p_{y_1}$ and $\p_t$ derivative. Hence the bigger temporal power will be applied to $\p_{y_1} \p_t Y$ and $\p_{y_1}^2 Y$.



Let us present the notation we shall be using. For any $1\le p\le \infty$ and any measurable scalar or vector function $f$,
we will use $\|f\|_{L^p}$ to denote the usual $L^p$ norm. 
We use $\|\cdot\|_{L^{p}_{y'}L^{q}_{y_1}}$ to denote the $L^{q}_{y_1}$ norm with respect to $y_1$ and the $L^p_{y'}$ norm with respect to $y_2$ and $y_3$.
For nonnegative integer $s$, the $H^s$ inner product denotes $(f|g)_{H^s}=\sum_{|\alpha|\leq s}\int_{\bR^3} \p^\alpha f\cdot\p^\alpha g\, \mathrm{d}y$.
For any two quantities $X$ and $Y$, we denote $X \lesssim Y$ if
$X \le C Y$ for some constant $C>0$. Similarly $X \gtrsim Y$ if $X
\ge CY$ for some $C>0$.  The dependence of the constant $C$ on
other parameters or constants are usually clear from the context and
we usually suppress  this dependence.

The rest of this paper is organized as follows. In Section \ref{secL}, we set up the Lagrangian reformulation  of the  MHD system \eqref{A1}. In Section \ref{seclinear} and Section \ref{secnonlinear}, we present the energy estimates for the solutions. The proof of global well-posedness result  under the Lagrangian coordinates is given in Section \ref{secproof1}.  The final section is devoted to the proof of the main theorem.

\section{Lagrangian formulation of the MHD equations}\label{secL}

In this section, let us show the Lagrangian formulation of the MHD equations \eqref{A3}.

Firstly, for any function $f(t,X(t,y))$, by the chain rule, we have
\begin{equation*}
	\f{\p f(t,X(t,y))}{\p y_i}=\left(\f{\p f}{\p x_k}\right)(t,X(t,y))\frac{\partial X^k}{\partial y_i}.
\end{equation*}
Then using the definition of $A$ in \eqref{dfnA}, we see
\begin{equation*}
	\bigl(\nabla_x f\bigr)(t,X(t,y))=A(t,y)\nabla_yf(t,X(t,y)).
\end{equation*}
Meanwhile, it follows from the equation of $b$ that
\begin{align}\label{eqb}
	\frac{\mathrm{d}}{\mathrm{d}t}b^i(t,X(t,y))=b^k(t,X(t,y))A_{kl}\partial_{y_{_l}}u^i(t,X(t,y)).
\end{align}
Left multiplying $A^\top$ to \eqref{eqb}, we get
\begin{align}\label{eqb1}
	A_{ij}\frac{\mathrm{d}}{\mathrm{d}t}b^i(t,X(t,y))=b^k(t,X(t,y))A_{kl}A_{ij}\partial_{y_{_l}}u^i(t,X(t,y)).
\end{align}
Since $A^\top \nabla_y X=I$, i.e., $A_{ij}\partial_{y_{_l}}X^i=\delta_{jl}$, one has
\begin{align*}
	\frac{\mathrm{d}}{\mathrm{d}t} A_{ij}\partial_{y_{_l}}X^i+A_{ij}\frac{\mathrm{d}}{\mathrm{d}t} \partial_{y_{_l}}X^i=0,
\end{align*}
which yields
\begin{align}\label{eqA}
	\frac{\mathrm{d}}{\mathrm{d}t} A_{ij}\partial_{y_{_l}}X^i+A_{ij}\partial_{y_{_l}}u^i(t,X(t,y))=0.
\end{align}
Combining \eqref{eqb1} and \eqref{eqA}, and noting that $A_{kl}\partial_{y_{_l}}X^i=\delta_{ki}$, we obtain
\begin{align*}
	\frac{\mathrm{d}}{\mathrm{d}t}\left(A_{ij}b^i(t,X(t,y))\right)=0.
\end{align*}
This implies
\begin{align}\label{eqAb0}
	A^\top(t,y) b(t,X(t,y))=A^\top(0,y) b(0,X(0,y))=A^\top_0(y) b_0(X_0(y)).
\end{align}
Here $A^\top_0(y)=(\nabla_yX_0)^{-1}$.
 By \eqref{eq-com} and \eqref{eqAb0}, we then have
\begin{equation}\label{eqAbt}
	A^\top(t,y) b(t,X(t,y))=e_1.
\end{equation}
Hence,
\begin{equation}\label{eqbt}
	b(t,X(t,y))=\nabla_y Xe_1=\partial_{y_{_1}} X.
\end{equation}
As a consequence, by using \eqref{eqAbt} and \eqref{eqbt}, we get
\begin{align*}
	\bigl(b^j\p_jb^i\bigr)(t,X(t,y))=
	&b^j(t,X(t,y))A_{jl}(t,y)\f{\p b^i(t,X(t,y))}{\p y_{_l}}\\
	=&\partial_{y_{_1}}b^i(t,X(t,y))=\p_{y_{_1}}^2 X^i(t,y).
\end{align*}
This yields the first equation in \eqref{A3}. Next we show the derivation of the  second one in \eqref{A3}.
Let $J(t,y)=\det (\nabla_y X)$, then
it is easy to see from $\nabla \cdot u=0$ that
\begin{equation*}
	\partial_t J=JA_{ij}\partial_{y_{j}}u^{i}(t,X(t,y))=0.
\end{equation*}
Therefore, $\det (\nabla_y X)=\det (\nabla_y X_0)=1$.

Now let us focus on the equation \eqref{A3} in Lagrangian coordinates.
Consider the solutions near equilibrium $X(t,y)=y+Y(t,y)$. Equation \eqref{A3} reduces to
\begin{equation}\label{A8}
\begin{cases}
Y_{tt} - \hbox{div}_y(A^\top A\nabla_y Y_t ) - \p_{y_{_1}}^2 Y + (I+\nabla_y Y)^{- \top}\nabla_y p=0,\\
\det(I+\nabla_y Y)=1,\\
Y(0,y)=Y_0(y),\quad Y_t(0,y)=Y_1(y),
\end{cases}
\end{equation}
where $Y_0(y)=X_0(y)-y$ and $Y_1(y)=u_0(X_0(y))$.
We rewrite \eqref{A8} in another form
\begin{equation}\label{A10}
\begin{cases}
 Y_{tt}- \Delta_y Y_t - \p_{y_{_1}}^2 Y=f,\\
\det(I+\nabla_y Y)=1\\
Y(0,y)=Y_0(y),\quad Y_t(0,y)=Y_1(y),
\end{cases}
\end{equation}
with
\begin{align*}
f=\hbox{div}_y\big( (A^\top A-I)\nabla_y Y_t \big)-(I+\nabla_y Y)^{-\top}\nabla_y p.
\end{align*}

In what follows, we derive the expression for the pressure under Lagrangian coordinates. Clearly,
in Eulerian coordinates, we have
\begin{align*}
-\Delta_x p(t,x)
=\sum_{i,j=1}\nabla_{x_i}\nabla_{x_j}\bigl(u^i u^j-b^i b^j \bigr).
\end{align*}
Denote $\nabla_Y=A \nabla_y$.  Direct calculation implies that
\begin{align*}
		-\nabla_Y\cdot  \nabla_Y p(t, X(t,y))
		&=\sum_{i,j} \nabla_{Y^i} \nabla_{Y^j}\bigl(X^i_t X^j_t-\p_{y_{_1}} X^i \p_{y_{_1}}X^j\bigr)(t,y).
\end{align*}
Since $X(t,y)=y+Y(t,y)$, we then infer that
\begin{align*}
\sum_{i,j}\nabla_{Y^i}\nabla_{Y^j}\bigl(\p_{y_{_1}}X^i\p_{y_{_1}}X^j\bigr)
&=\sum_{i,j}\nabla_{Y^i}\nabla_{Y^j}\bigl((\delta_{1i}+\p_{y_{_1}}Y^i)
(\delta_{1j}+\p_1Y^j)\bigr)\\\nonumber
&=\sum_{i,j}\nabla_{Y^i}\nabla_{Y^j}\bigl(\p_{y_{_1}} Y^i\p_{y_{_1}} Y^j\bigr)
+2 \sum_{i,j}\nabla_{Y^i}\nabla_{Y^j}\bigl(\delta_{1i}\p_{y_{_1}} Y^j\bigr).
\end{align*}
In Eulerian coordinates, the magnetic field satisfies the divergence free condition $\nabla_x \cdot b(t,x)=0$.
Thus in Lagrangian coordinates, we have
\begin{align*}
0=\nabla_Y \cdot b(t,X(t,y))=\nabla_Y \cdot \p_{y_{_1}} X
= \nabla_Y \cdot \big( \p_{y_{_1}} Y+e_1\big) = \nabla_Y \cdot   \p_{y_{_1}} Y.
\end{align*}
Consequently,
\begin{align}\label{eqp1}
-\nabla_Y\cdot \nabla_Y p(t, X(t,y))
=\sum_{i,j} \nabla_{Y^i} \nabla_{Y^j}\bigl(Y^i_t Y^j_t-\p_{y_{_1}} Y^i \p_{y_{_1}}Y^j\bigr)(t,y).
\end{align}
For the left hand side of \eqref{eqp1}, we have
\begin{align}\label{eqp2}
-\nabla_Y\cdot \nabla_Y p(t, X(t,y))	&=-\textrm{div}_y( A^\top A \nabla_y p(t, X(t,y))).
\end{align}
This yields that
\begin{align}\label{eqp3}
	\begin{split}
			&\sum_{i,j} \nabla_{Y^i} \nabla_{Y^j}\bigl(Y^i_t Y^j_t-\p_{y_{_1}} Y^i \p_{y_{_1}}Y^j\bigr)(t,y)\\&\quad=\operatorname{div}_{y}\Bigl(A^\top\operatorname{div}_{y}\big(A^\top(Y^i_t Y^j_t-\p_{y_{_1}} Y^i \p_{y_{_1}}Y^j)\big)\Bigr)(t,y).
	\end{split}
\end{align}
Combining \eqref{eqp1}, \eqref{eqp2} and \eqref{eqp3}, we deduce
\begin{align*}
	\begin{split}
		p(t,X(t,y))&=-\Delta_y^{-1}\textrm{div}_y\big(( A^\top A-I) \nabla_y p(t, X(t,y))\big)\\&\quad+\Delta_y^{-1}\operatorname{div}_{y}\Bigl(A^\top\operatorname{div}_{y}\big(A^\top(\p_{y_{_1}} Y^i \p_{y_{_1}}Y^j-Y^i_t Y^j_t)\big)\Bigr)(t,y).
	\end{split}
\end{align*}
Thus we finish the derivation of the pressure $p$ under the Lagrangian coordinates.
\section{Estimate of the linear system}\label{seclinear}
In Section \ref{seclinear} and Section \ref{secnonlinear}, we present the energy estimate for \eqref{A10}.
Let $Y$ be a sufficiently smooth  solution of \eqref{A10} on
$[0,T).$
In this section,
we  cook up the energy estimates for the linear system
\begin{equation}\label{B19}
	\begin{cases}
		Y_{tt}-\Delta Y_t-\pa^2_1Y=f, \\
		Y|_{t=0}=Y_0,\quad Y_t|_{t=0}=Y_1.
	\end{cases}
\end{equation}
 The main result is stated in the following lemma.
\begin{lem}\label{Lem1}
	Let $Y$ be a smooth  solution of \eqref{B19} on $[0,T).$
	Then for all $t\in [0,T),$ there holds
	\begin{align}\label{mathcalE}
		\begin{split}
				\mathcal{E}(t)&\lesssim \left(\|  Y_1\|_{H^3}^2
			+\|\pa_1Y_0\|_{H^3}^2
			+\|  \D Y_0\|_{H^2}^2\right)\\&\quad
			+ \int_0^t \big|( f | Y_\tau -\frac14 \D Y -\f{1}{4}(\tau+1) \D Y_\tau)_{H^2}\big|\,\mathrm{d}\tau\\&\quad+\int_0^t\big|( f | \f{1}{16} (\tau+1)\D\p_1^2 Y+\f{1}{32} (\tau+1)^2 \D\p_1^2 Y_\tau )_{H^1} \big|\,\mathrm{d}\tau.
		\end{split}
	\end{align}
\end{lem}
\begin{proof}
	\textbf{Step One.}
	Taking the $H^2$ inner product of \eqref{B19} with $ Y_t,$ we obtain the following identity
	\begin{equation}\label{S4eq10}
		\f12\f{\mathrm{d}}{\mathrm{d}t}\bigl(\|Y_t\|_{H^2}^2+\|\pa_{1}Y\|_{H^2}^2\bigr)
		+\|\na Y_t\|_{H^2}^2 =\bigl( f | Y_t\bigr)_{H^2}.
	\end{equation}
	Along the same line, taking the $H^2$ inner product of \eqref{B19} with $ -\frac14\D Y,$
	one has
	\begin{equation*}
		\f18\f{\mathrm{d}}{\mathrm{d}t}\|\D Y\|_{H^2}^2+\frac14\|\pa_1\na  Y\|_{H^2}^2
		-\frac14( Y_{tt} |\D Y)_{H^2}=-\frac14( f | \D Y)_{H^2}.
	\end{equation*}
	Notice that
	\[ (Y_{tt} | \D  Y)_{H^2}
	=\f{\mathrm{d}}{\mathrm{d}t}(  Y_{t}|
	\D  Y)_{H^2} +\|\na  Y_{t}\|_{H^2}^2.\]
	Hence \begin{equation}\label{d4}
		\begin{split}
			\f14 \f{\mathrm{d}}{\mathrm{d}t}\Bigl(\f12\|\D  Y\|_{_{H^2}}^2
			- (Y_{t} |\D  Y)_{H^2} \Bigr)
			-\f14 \|\na  Y_t\|_{H^2}^2
			+\f14 \|\pa_1\na Y\|_{H^2}^2 =-\f14 (  f |\D  Y)_{H^2}.
		\end{split}
	\end{equation}
	By summing up \eqref{S4eq10} with
	\eqref{d4}, we obtain
	\begin{equation}\label{d5}
		\begin{split}
			\f{\mathrm{d}}{\mathrm{d}t}  \Bigl( \f12\|  Y_t(t)\|_{H^2}^2
			&+\f12 \|  \pa_1Y(t)\|_{H^2}^2
			+\f18\|  \D Y(t)\|_{H^2}^2
			-\f14\bigl(  Y_t(t) |   \D Y(t)\bigr)_{H^2} \Bigr) \\
			& +\f34\|\na  Y_t\|_{H^2}^2
			+\f14\|\pa_1\na  Y\|_{H^2}^2
			=\bigl(  f |
			Y_t-\f14\D  Y\bigr)_{H^2}.
		\end{split}
	\end{equation}

	\textbf{Step Two.}
	Taking the $H^2$ inner product of \eqref{B19} with $-\f 14 (t+1) \D Y_t,$ the elementary calculation implies that
	\begin{equation}\label{d6}
		\begin{split}
			&\f18\f{\mathrm{d}}{\mathrm{d}t} \bigl((t+1)\|  \nabla Y_t (t)\|_{H^2}^2
			+(t+1)\|  \nabla \pa_1Y(t)\|_{H^2}^2  \bigr) \\
			&\quad- \f 18\bigl(\|  \nabla Y_t (t)\|_{H^2}^2
			+\|  \nabla \pa_1Y(t)\|_{H^2}^2  \bigr)
			+\f 14 (t+1) \|\D Y_t\|_{H^2}^2 =-\f 14 (t+1)  ( f | \D Y_t )_{H^2}.
		\end{split}
	\end{equation}
	
	
	\textbf{Step Three.}
	Similarly, taking the $H^1$ inner product of \eqref{B19} with $ \f{1}{16} (t+1) \D\p_1^2 Y,$ we get that
	\begin{equation}\label{d7}
		\begin{aligned}
			 \f{\mathrm{d}}{\mathrm{d}t} &\Big( \f{1}{32}(t+1)\|  \D\p_1 Y (t)\|_{H^1}^2
			-\f{1}{16} (t+1) (\D \p_1 Y| \p_1 Y_t)_{H^1}
			-\f{1}{32}  \| \nabla\p_1 Y (t)\|_{H^1}^2 \Big) \\
			&-\f{1}{32}  \|  \D\p_1 Y (t)\|_{H^1}^2
			-\f{1}{16} (t+1) \| \nabla\p_1 Y_t \|^2_{H^1}
			+\f{1}{16} (t+1) \|\nabla \p_1^2Y \|^2_{H^1} \\
			& =\f{1}{16}(t+1) ( f |  \D\p_1^2 Y )_{H^1}.
		\end{aligned}
	\end{equation}
	
	\textbf{Step Four.}
	At the last step, by taking the $H^1$ inner product of \eqref{B19} with $ \f{1}{32} (t+1)^2 \D\p_1^2 Y_t,$ we  have
	\begin{equation}\label{d8}
		\begin{aligned}
			&
			\f{1}{64} \f{\mathrm{d}}{\mathrm{d}t} \left((t+1)^2 \|  \nabla\p_1 Y_t (t)\|_{H^1}^2
			+ (t+1)^2 \|  \nabla \pa_1^2 Y \|_{H^1}^2  \right) \\
			&-\f{1}{32} (t+1)  \bigl(\|  \nabla\p_1Y_t (t)\|_{H^1}^2
			+ \|  \nabla \pa_1^2 Y \|_{H^1}^2  \bigr)
			+\f{1}{32} (t+1)^2 \| \D\p_1 Y_t\|_{H^1}^2 \\
			&=\f{1}{32} (t+1)^2  ( f |  \D\p_1^2 Y_t )_{H^1}.
		\end{aligned}
	\end{equation}
	
	By summing up  \eqref{d5}, \eqref{d6}, \eqref{d7}, \eqref{d8} and canceling the dissipative energy with negative sign, we deduce that
	\begin{equation}\label{d9}
		\begin{split}
			\f{\mathrm{d}}{\mathrm{d}t}  \tilde{E}(t) &+\f58\|\na  Y_t\|_{H^2}^2
			+\frac{3}{32}\|\na\pa_1  Y\|_{H^2}^2 +\f {1}{16} (t+1) \|\D Y_t\|_{H^2}^2
			\\&+\f{1}{32} (t+1) \|\nabla\p_1^2 Y \|^2_{H^1}
			+\f{1}{32} (t+1)^2 \| \D\p_1 Y_t\|_{H^1}^2 \\
			&\quad\quad\quad\leq \big|\bigl( f |  Y_t-\frac14 \D Y - \f{1}{4}(t+1)\D Y_t\bigr)_{H^2}\big|\\
&\quad\quad\quad\quad
			+\big|\bigl( f |  \f{1}{16} (t+1)\D\p_1^2 Y+\f{1}{32} (t+1)^2 \D\p_1^2 Y_t \bigr)_{H^1}\big|,
		\end{split}
	\end{equation}
	where
	\begin{align*}
		\tilde{E}(t)&:=\f12\|  Y_t(t)\|_{H^2}^2
		+\f12 \|  \pa_1Y(t)\|_{H^2}^2
		+\f18\|  \D Y(t)\|_{H^2}^2
		-\f14\bigl(  Y_t(t) |   \D Y(t)\bigr)_{H^2}  \\
		&\quad +\f18 (t+1)\|  \nabla Y_t (t)\|_{H^2}^2
		+ \f18 (t+1)\|  \nabla \pa_1Y(t)\|_{H^2}^2   \\
		&\quad + \f{1}{32} (t+1)\|  \D\p_1 Y (t)\|_{H^1}^2
		-\f{1}{16} (t+1)(\D\p_1 Y| \p_1 Y_t)_{H^1}
		-\f{1}{32}  \| \nabla\p_1 Y (t)\|_{H^1}^2  \\
		&\quad +  \f{1}{64} (t+1)^2 \|  \nabla\p_1 Y_t (t)\|_{H^1}^2
		+ \f{1}{64} (t+1)^2 \|  \nabla \pa_1^2 Y \|_{H^1}^2.
	\end{align*}
	By using the H\"older and Young inequalities, there hold
	\begin{align*}
		\f14\bigl(  Y_t |   \D Y(t)\bigr)_{H^2}\leq \frac{1}{16}\|\Delta Y\|_{H^2}^2+\frac{1}{4}\|Y_t\|_{H^2}^2
	\end{align*}
	and
	\begin{align*}
		\f{1}{16} (t+1)(\D\p_1 Y| \p_1 Y_t)_{H^1}\leq \frac{1}{64}(t+1)\|\Delta\p_1 Y\|_{H^1}^2+\frac{1}{16}(t+1)\|\p_1Y_t\|_{H^1}^2.
	\end{align*}
	Hence,
	we deduce the following lower bound for $\tilde{E}(t)$
	\begin{equation} \label{d10}
		\begin{split}
			\tilde{E}(t)&\geq
			\f14\|  Y_t(t)\|_{H^2}^2
			+\f12 \|  \pa_1Y(t)\|_{H^2}^2
			+\f{1}{32}\|  \D Y(t)\|_{H^2}^2 +\f{1}{16} (t+1)\|  \nabla Y_t (t)\|_{H^2}^2 \\
			&\quad\quad\quad\quad\quad\quad\
			+ \f{1}{16} (t+1)\|  \nabla \pa_1Y(t)\|_{H^2}^2+ \f{1}{64} (t+1)\|  \D\p_1 Y (t)\|_{H^1}^2   \\
			&\quad\quad\quad\quad\quad\quad\
			+   \f{1}{64} (t+1)^2 \|  \nabla\p_1 Y_t (t)\|_{H^1}^2
			+ \f{1}{64} (t+1)^2 \|  \nabla \pa_1^2 Y \|_{H^1}^2  .
		\end{split}
	\end{equation}
	At the same time, it is easy to get the upper bound of $\tilde{E}(0)$, that is
	\begin{equation} \label{d11}
		\tilde{E}(0)\lesssim
		\|  Y_1\|_{H^3}^2
		+\|  \pa_1Y_0\|_{H^3}^2
		+\|  \D Y_0\|_{H^2}^2.
	\end{equation}
For \eqref{d9}, performing the time integration over the interval $[0,t]$, together with \eqref{d10} and \eqref{d11}  will yield the lemma.
\end{proof}

\section{Estimates of the nonlinear terms}\label{secnonlinear}
In this section, we are going to handle the nonlinear terms. The goal is to control the second and third terms on the right hand side of \eqref{mathcalE} by $\mathcal{E}(t)$ under the ansatz that $\mathcal{E}(t)$ is sufficiently small.
The main result can be stated as follows.
\begin{lem}\label{Lem2}
	Let $Y$ be a smooth  solution of \eqref{B19} on $[0,T).$
	There exists a sufficiently small constant $\delta\in (0,1)$ such that if $\mathcal{E}(t)\leq \delta$,  then
	\begin{align*}
		\begin{split}
			&\int_0^t \big|( f | Y_\tau -\frac14 \D Y -\f{1}{4}(\tau+1) \D Y_\tau)_{H^2}\big|\,\mathrm{d}\tau\\&\quad+\int_0^t\big|( f | \f{1}{16} (\tau+1)\D\p_1^2 Y+\f{1}{32} (\tau+1)^2 \D\p_1^2 Y_\tau )_{H^1} \big|\,\mathrm{d}\tau\\&\lesssim \mathcal{E}(t)^\frac32+\mathcal{E}(t)^2
		\end{split}
	\end{align*}
holds for all $t\in [0,T)$.
\end{lem}
\begin{rem}
	The key point of the estimate of the nonlinear terms
 is to obtain the $L^1$ integrability in time.
	The trouble mainly comes from the terms
containing 	
$\nabla Y_t$.
This motivates us to cook up the temporal weighted energy.
\end{rem}

The proof of Lemma \ref{Lem2} will be divided into three parts: the quadratic nonlinearities, the higher order nonlinearities and the pressure term.
They will be treated in the following three subsections.

Before going any further, we first study the structure of nonlinear terms.
Recall that
\begin{align*}
	f=\hbox{div}\big( (A^\top A-I)\nabla Y_t \big)-A\nabla p,
\end{align*}
with
\begin{equation}\label{exp-A}
	A=\big(I+\na Y\big)^{-\top},
\end{equation}
and
\begin{align*}
	\nabla p&=-\Delta^{-1}\nabla \textrm{div}\big(( A^\top A-I) \nabla p\big)\\&\quad+\Delta^{-1}\nabla \operatorname{div}\Bigl(A^\top\operatorname{div}\big(A^\top(\p_1 Y\otimes \p_1Y-Y_t\otimes  Y_t)\big)\Bigr).
\end{align*}
Due to the definition of $A$ in \eqref{exp-A},
it is easy to see
\begin{align*}
	A=A^*/\det(I+\nabla Y),
\end{align*}
where $A^*=(A^*_{ij})_{3\times 3}$ and $A^*_{ij}$ is the algebraic complement minor of the $(i,j)$th entry in the matrix $(I+\nabla Y)$.
Since $\det(I+\nabla Y)=1$, we obtain the following expression of $A$
\begin{align}\label{expA}
	A=I+B_1+B_2,
\end{align}
with
\begin{align*}
	B_1=I\dv Y-(\na Y)^\top
\end{align*}
and
\begin{align*}
	B_2&=
	\left(
	\begin{array}{ccc}
		\p_2Y^2\p_3Y^3-\p_2Y^3\p_3Y^2 & \p_1Y^3\p_3Y^2-\p_1Y^2\p_3Y^3 & \p_1Y^2\p_2Y^3-\p_1Y^3\p_2Y^2\\ \p_2Y^3\p_3Y^1-\p_2Y^1\p_3Y^3
		&	\p_1Y^1\p_3Y^3 -\p_1Y^3\p_3Y^1   &\p_1Y^3\p_2Y^1-\p_2Y^3\p_1Y^1\\
		\p_2Y^1\p_3Y^2-\p_2Y^2\p_3Y^1 &\p_1Y^2\p_3Y^1-\p_1Y^1\p_3Y^2 & \p_1Y^1\p_2Y^2-\p_1Y^2\p_2Y^1
	\end{array}
	\right).
\end{align*}
According to the fact that
\begin{align*}
	A^\top A-I&=(A^\top-I)+(A-I)+(A^\top-I)(A-I)\\
	&=(B_1^\top+B_2^\top)+(B_1+B_2)+(B_1^\top+B_2^\top)(B_1+B_2),
\end{align*}
we can roughly treat $A^\top A-I$ as
\begin{equation}\label{exp-A2}
O\Big(\nabla Y+(\nabla Y)^2+(\nabla Y)^3+(\nabla Y)^4\Big).
\end{equation}
Here, terms containing $\nabla Y$ and $(\nabla Y)^2$ come from $B_1$ or $B_1^\top$ and $B_2$ or $B_2^\top$, respectively. Terms containing $(\nabla Y)^3$ and $(\nabla Y)^4$ come from $(B_1^\top+B_2^\top)(B_1+B_2)$.
For simplicity, we roughly take $f$ as
\begin{align}\label{exp-f}
	f\sim \nabla(\nabla Y\nabla Y_t) +\nabla\Big(\big((\nabla Y)^2+(\nabla Y)^3+(\nabla Y)^4\big)\nabla Y_t\Big)
	+A\nabla p.
\end{align}
In the following subsections, we are going to  estimate the nonlinear terms $f$ in \eqref{exp-f} one by one.
\subsection{Estimates of quadratic nonlinear terms}\label{quad}
In this subsection, we  deal with the first term in \eqref{exp-f}
\begin{align}\label{f12-d}
\nabla(\nabla Y\nabla Y_t)
=\nabla^2 Y\nabla Y_t+\nabla Y\nabla^2 Y_t:=f_1+f_2.
\end{align}
By integration by parts, we organize
\begin{align}\label{f12}
&(f_1+f_2| Y_t-\frac{1}{4}\Delta Y-\f{1}{4}(t+1) \D Y_t)_{H^2}
+(f_1+f_2|\frac{1}{16}(t+1)\partial_1^2\Delta Y+\frac{1}{32}(t+1)^2\partial_1^2\Delta Y_t)_{H^1},\nonumber\\\nonumber
&=(f_1| Y_t-\frac{1}{4}\Delta Y-\f{1}{4}(t+1) \D Y_t)_{H^2}
+(f_1|\frac{1}{16}(t+1)\partial_1^2\Delta Y)_{H^1}
+(f_1|\frac{1}{32}(t+1)^2\partial_1^2\Delta Y_t)_{H^1}\\\nonumber
&\quad+(f_2| Y_t-\frac{1}{4}\Delta Y-\f{1}{4}(t+1) \D Y_t)_{H^2}
+(f_2|\frac{1}{16}(t+1)\partial_1^2\Delta Y)_{H^1}
+(f_2|\frac{1}{32}(t+1)^2\partial_1^2\Delta Y_t)_{H^1}\\
&=I_1+I_2+I_3+J_1+J_2+J_3.
\end{align}
Before presenting the estimates of these terms, we prepare the bounds for
$\|f_1\|_{H^2}$, $\|\p_1 f_1\|_{H^1}$, $\|f_2\|_{H^2}$ and $\|\p_1 f_2\|_{H^1}$
by $E(t)$ and $D(t)$.
\begin{lem}
There holds
\begin{align}
\label{esfH2}
\|f_1\|_{H^2}&\lesssim (t+1)^{-\frac{7}{12}}E(t)^\frac12D(t)^\frac12,\\
\label{esp1f1H1}
\|\p_1f_1\|_{H^1}&\lesssim (t+1)^{-1}E(t)^\frac12D(t)^\frac12,\\
\label{esf2H2}
\|f_2\|_{H^2}&\lesssim(t+1)^{-\frac{7}{12}}E(t)^\frac12D(t)^\frac12,\\
\label{esp1f2H1}
\|\partial_1 f_2\|_{H^1}&\lesssim (t+1)^{-1}E(t)^\frac12D(t)^\frac12.
\end{align}
\end{lem}
\begin{proof}
\textbf{Estimate of $\|f_1\|_{H^2}$}.

Using the interpolation inequality and the anisotropic Sobolev inequality, one has
\begin{align}\label{es2}
	\begin{split}
		\|\nabla Y_t\|_{L^\infty}&\lesssim\|\nabla^2 Y_t\|^{\frac{1}{2}}_{L^2}\|\nabla^2 Y_t\|^{\frac{1}{2}}_{L^6}\\
		&\lesssim\|\nabla^2 Y_t\|^{\frac{1}{2}}_{L^2}\big\|\|\nabla^2 Y_t\|^{\frac{2}{3}}_{L^2_{y_1}}\|\partial_1\nabla^2 Y_t\|^{\frac{1}{3}}_{L^2_{y_1}}\big\|_{L^6_{y'}}^{\frac{1}{2}}\\
		&\lesssim\|\nabla^2 Y_t\|_{H^1}^{\frac{5}{6}}\|\partial_1\nabla^2 Y_t\|_{H^1}^{\frac{1}{6}}.
	\end{split}
\end{align}
Then, by the H\"older inequality and \eqref{es2}, we obtain the estimate of $\|f_1\|_{L^2}$ by
\begin{align}\label{esf1}
	\|f_1\|_{L^2}
	\lesssim \|\nabla^2 Y\|_{L^2}\|\nabla Y_t\|_{L^\infty}\lesssim \|\nabla^2 Y\|_{L^2}\|\nabla^2 Y_t\|_{H^1}^{\frac{5}{6}}\|\partial_1\nabla^2 Y_t\|_{H^1}^{\frac{1}{6}}.
\end{align}
On the other hand, for $\|\nabla^2 f_1\|_{L^2}$, we have
\begin{align}\label{es3}
	\|\nabla^2 f_1\|_{L^2}
	\lesssim \|\nabla^4Y\nabla Y_t\|_{L^2}+\|\nabla^3Y\nabla^2 Y_t\|_{L^2}+\|\nabla^2Y\nabla^3 Y_t\|_{L^2}.
\end{align}
By the H\"older inequality and \eqref{es2}, we get
\begin{align}\label{es4}
	\begin{split}
		\|\nabla^4Y\nabla Y_t\|_{L^2}\leq \|\nabla^4Y\|_{L^2}\|\nabla Y_t\|_{L^\infty}\lesssim \|\nabla^4Y\|_{L^2}\|\nabla^2 Y_t\|_{H^1}^{\frac{5}{6}}\|\partial_1\nabla^2 Y_t\|_{H^1}^{\frac{1}{6}}.
	\end{split}
\end{align}
Applying the anisotropic H\"older and the Sobolev inequalities, we derive that
\begin{align}\label{es5}
	\begin{split}
		&\|\nabla^3Y\nabla^2 Y_t\|_{L^2}+\|\nabla^2Y\nabla^3 Y_t\|_{L^2}\\&\lesssim  \|\nabla^3Y\|_{L^2_{y'}L^\infty_{y_1}}\|\nabla^2Y_t\|_{L^\infty_{y'} L^2_{y_1}}+\|\nabla^2Y\|_{L^\infty_{y'}L^2_{y_1}}\|\nabla^3Y_t\|_{L^2_{y'}L^\infty_{y_1}}\\&\lesssim \|\nabla^3Y\|_{L^2}^{\frac12}\|\partial_1\nabla^3Y\|_{L^2}^\frac12\|\nabla^2Y_t\|_{H^2}+\|\nabla^2Y\|_{H^2}\|\nabla^3Y_t\|_{L^2}^\frac12\|\partial_1\nabla^3Y_t\|_{L^2}^\frac12.
	\end{split}
\end{align}
Plugging \eqref{es4} and \eqref{es5} into \eqref{es3}, one has
\begin{align}\label{esna2f1}
	\begin{split}
		\|\nabla^2 f_1\|_{L^2}
		&\lesssim \|\Delta Y\|_{H^2}\|\Delta Y_t\|_{H^2}^{\frac{5}{6}}\|\partial_1\Delta Y_t\|_{H^1}^{\frac{1}{6}}+ \|\Delta Y\|_{H^2}^{\frac12}\|\partial_1\nabla Y\|_{H^2}^\frac12\|\Delta Y_t\|_{H^2}\\&\quad+\|\Delta Y\|_{H^2}\|\Delta Y_t\|_{H^2}^\frac12\|\partial_1\Delta Y_t\|_{H^1}^\frac12.
	\end{split}
\end{align}
Combining \eqref{esf1}  with \eqref{esna2f1}, we obtain
\begin{align*}
	\begin{split}
		\|f_1\|_{H^2}
		&\lesssim (t+1)^{-\frac{7}{12}}\|\Delta Y\|_{H^2}\big(\sqrt{t+1}\|\Delta Y_t\|_{H^2}\big)^{\frac{5}{6}}\big({(t+1)}\|\partial_1\Delta Y_t\|_{H^1}\big)^{\frac{1}{6}}\\&\quad+(t+1)^{-\frac{3}{4}}\|\Delta
		Y\|_{H^2}^{\frac{1}{2}}\big(\sqrt{t+1}\|\p_1\nabla  Y\|_{H^2}\big)^\frac12\big(\sqrt{(t+1)}\|\Delta Y_t\|_{H^2}\big)\\&\quad+(t+1)^{-\frac34}\|\Delta Y\|_{H^2}\big(\sqrt{t+1}\|\Delta Y_t\|_{H^2}\big)^{\frac{1}{2}}\big({(t+1)}\|\partial_1\Delta Y_t\|_{H^1}\big)^{\frac{1}{2}}\\&\lesssim (t+1)^{-\frac{7}{12}}E(t)^\frac12D(t)^\frac12.
	\end{split}
\end{align*}

\textbf{Estimate of $\|\p_1f_1\|_{H^1}$}.

Now we estimate $\|\p_1f_1\|_{H^1}$. Let us first consider $\|\partial_1f_1\|_{L^2}$. Note that
\begin{align*}
	\partial_1f_1=\partial_1\nabla^2 Y\nabla Y_t+\nabla^2 Y\partial_1\nabla Y_t.
\end{align*}
By the H\"older and Sobolev inequalities, we have
\begin{align}\label{es6}
	\|\partial_1\nabla Y_t\|_{L^\infty}\lesssim \|\partial_1\nabla^2 Y_t\|_{L^2}^\frac12\|\partial_1\nabla^2 Y_t\|_{L^6}^\frac12\lesssim \|\partial_1\Delta Y_t\|_{H^1}.
\end{align}
Thus, according to \eqref{es2} and \eqref{es6}, we derive
\begin{align}
	\begin{split}\label{esp1f1}
		\|\partial_1 f_1\|_{L^2}&\lesssim\|\partial_1\nabla^2 Y\|_{L^2}\|\nabla Y_t\|_{L^\infty}+\|\nabla^2 Y\|_{L^2}\|\partial_1\nabla Y_t\|_{L^\infty}\\
		&\lesssim \|\partial_1\nabla^2  Y\|_{L^2}\|\Delta Y_t\|_{H^1}^{\frac{5}{6}}\|\partial_1\Delta Y_t\|_{H^1}^{\frac{1}{6}}+\|\nabla^2 Y\|_{L^2}\|\partial_1\Delta Y_t\|_{H^1}.
	\end{split}
\end{align}

On the other hand, we compute
\begin{equation*}
	\partial_1\nabla f_1= \partial_1\nabla^3Y\nabla Y_t+\nabla^3Y\partial_1\nabla Y_t+\partial_1\nabla^2 Y\nabla^2Y_t+\nabla^2Y\partial_1\nabla^2Y_t.
\end{equation*}
Using the anisotropic H\"older and Sobolev inequalities, we get
\begin{align}\label{es7}
	\begin{split}
		&\|\partial_1\nabla^2 Y\nabla^2Y_t\|_{L^2}
		+\|\nabla^2Y\partial_1\nabla^2Y_t\|_{L^2}\\
&\lesssim
\|\partial_1\nabla^2 Y\|_{L^2_{y'}L^\infty_{y_1}}\|\nabla^2Y_t\|_{L^\infty_{y'} L^2_{y_1}}
+\|\nabla^2Y\|_{L^4}\|\partial_1\nabla^2Y_t\|_{L^4}\\
&\lesssim \|\partial_1\nabla^2 Y\|_{L^2}^{\frac12}\|\partial_1^2\nabla^2Y\|_{L^2}^\frac12\|\nabla^2Y_t\|_{H^2}+\|\nabla^2Y\|_{H^2}\|\partial_1\Delta Y_t\|_{H^1}.
	\end{split}
\end{align}
It then follows from \eqref{es2}, \eqref{es6} and \eqref{es7} that
\begin{align}
	\begin{split}\label{esp1naf1}
		\|\partial_1 \nabla f_1\|_{L^2}&\lesssim  \|\partial_1\nabla^3 Y\|_{L^2}\|\nabla Y_t\|_{L^\infty}+\|\nabla^3 Y\|_{L^2}\|\partial_1\nabla Y_t\|_{L^\infty}\\
		&\quad+\|\partial_1\nabla^2 Y\nabla^2Y_t\|_{L^2}
		+\|\nabla^2Y\partial_1\nabla^2Y_t\|_{L^2}\\
		&\lesssim \|\partial_1\nabla^3  Y\|_{L^2}\|\Delta Y_t\|_{H^1}^{\frac{5}{6}}\|\partial_1\Delta Y_t\|_{H^1}^{\frac{1}{6}}+\|\nabla^3 Y\|_{L^2} \|\partial_1\Delta Y_t\|_{H^1}\\
		&\quad+\|\partial_1\nabla^2 Y\|_{L^2}^{\frac12}\|\partial_1^2\nabla^2Y\|_{L^2}^\frac12\|\nabla^2Y_t\|_{H^2}+\|\nabla^2Y\|_{H^2}\|\partial_1\Delta Y_t\|_{H^1}.
	\end{split}
\end{align}
Combining \eqref{esp1f1} with \eqref{esp1naf1}, we infer
\begin{align*}
	\begin{split}
		\|\p_1f_1\|_{H^1}&\lesssim (t+1)^{-\frac{13}{12}}\big(\sqrt{t+1}\|\p_1\nabla  Y\|_{H^2}\big)\big(\sqrt{t+1}\|\Delta Y_t\|_{H^2}\big)^{\frac{5}{6}}\big({(t+1)}\|\partial_1\Delta Y_t\|_{H^1}\big)^{\frac{1}{6}}\\
&\quad+(t+1)^{-\frac{5}{4}}\big(\sqrt{t+1}\|\p_1\nabla  Y\|_{H^2}\big)^\frac12\big((t+1)\|\p_1^2\nabla  Y\|_{H^1}\big)^\frac12\big(\sqrt{t+1}\|\Delta Y_t\|_{H^2}\big)\\&\quad+(t+1)^{-1}\|\D Y\|_{H^2}\big({(t+1)}\|\partial_1\Delta Y_t\|_{H^1}\big)\\&\lesssim (t+1)^{-1}E(t)^\frac12D(t)^\frac12.
	\end{split}
\end{align*}
\textbf{Estimate of $\|f_2\|_{H^2}$}.

To estimate $\|f_2\|_{H^2}$, we use the  anisotropic H\"older  and Sobolev inequalities to obtain
\begin{align}
	\begin{split}\label{esf2}
		\|f_2\|_{L^2}&\lesssim \|\nabla Y\|_{L^6}\|\nabla^2Y_t\|_{L^3} \\
&\lesssim \|\nabla^2 Y\|_{L^2}\big\|\|\nabla^2Y_t\|_{L^2_{y_1}}^\frac56\|\partial_1\nabla^2Y_t\|_{L^2_{y_1}}^\frac16\big\|_{L^3_{y'}}\\
&\lesssim\|\nabla^2 Y\|_{L^2}\|\nabla^2Y_t\|_{H^1}^\frac56 \|\partial_1\nabla^2Y_t\|_{L^2}^\frac16.
	\end{split}
\end{align}

Next, we estimate $\|\nabla^2f_2\|_{L^2}$. Note that
\begin{equation*}
	\nabla^2 f_2= \nabla^3 Y\nabla^2Y_t+2\nabla^2 Y\nabla^3Y_t+\nabla Y\nabla^4Y_t.
\end{equation*}
By  the H\"older inequality, anisotropic interpolation inequality and Sobolev inequality, we obtain
\begin{align}\label{es9}
	\begin{split}
		\|\nabla Y\|_{L^\infty}
		&\lesssim \|\nabla^2 Y\|_{L^2}^{\frac{1}{2}}\|\nabla^2 Y\|_{L^6}^{\frac{1}{2}}\\
		&\lesssim\|\nabla^2 Y\|_{L^2}^{\frac{1}{2}}\big\|\|\nabla^2 Y\|_{L^2_{y_1}}^\frac23\|\partial_1\nabla^2 Y\|_{L^2_{y_1}}^\frac13\big\|_{L^6_{y'}}^\frac12\\&\lesssim\|\Delta Y\|_{H^2}^{\frac{5}{6}}\|\partial_1\nabla Y\|_{H^2}^{\frac{1}{6}}.
	\end{split}
\end{align}
The combination of \eqref{es5}
and \eqref{es9} shows that
\begin{align}\label{esna2f2}
	\begin{split}
		\|\nabla^2 f_2\|_{L^2}&\lesssim 	\|\Delta Y\|_{H^2}^{\frac12}\|\partial_1\nabla Y\|_{H^2}^\frac12\|\Delta Y_t\|_{H^2}+\|\Delta Y\|_{H^2}\|\Delta Y_t\|_{H^2}^\frac12\|\partial_1\Delta Y_t\|_{H^1}^\frac12\\&\quad +\|\Delta Y\|_{H^2}^{\frac{5}{6}}\|\partial_1\nabla Y\|_{H^2}^{\frac{1}{6}}\|\Delta Y_t\|_{H^2}.
	\end{split}
\end{align}
It then follows from \eqref{esf2} and \eqref{esna2f2} that
\begin{align*}
	\begin{split}
		\| f_2\|_{H^2}&\lesssim (t+1)^{-\frac{7}{12}}\|\Delta Y\|_{H^2}\big(\sqrt{t+1}\|\Delta Y_t\|_{H^2}\big)^{\frac{5}{6}}\big({(t+1)}\|\partial_1\Delta Y_t\|_{H^1}\big)^{\frac{1}{6}}\\&\quad+(t+1)^{-\frac{3}{4}}\|\Delta Y\|_{H^2}\big(\sqrt{t+1}\|\Delta Y_t\|_{H^2}\big)^{\frac{1}{2}}\big({(t+1)}\|\partial_1\Delta Y_t\|_{H^1}\big)^{\frac{1}{2}}\\&\quad+(t+1)^{-\frac{3}{4}}\|\Delta Y\|_{H^2}^\frac12\big(\sqrt{t+1}\|\p_1\nabla  Y\|_{H^2}\big)^{\frac{1}{2}}\big(\sqrt{t+1}\|\Delta Y_t\|_{H^2}\big)\\&\quad+(t+1)^{-\frac{7}{12}}\|\Delta Y\|_{H^2}^\frac56\big(\sqrt{t+1}\|\p_1\nabla  Y\|_{H^2}\big)^{\frac{1}{6}}\big(\sqrt{t+1}\|\Delta Y_t\|_{H^2}\big)\\&\lesssim (t+1)^{-\frac{7}{12}}E(t)^\frac12D(t)^\frac12.
	\end{split}
\end{align*}
\textbf{Estimate of $\| \partial_1f_2\|_{H^1}$}.

It is clear that
\begin{align*}
	\partial_1f_2&= \partial_1\nabla Y \nabla^2 Y_t +\nabla Y\partial_1\nabla^2Y_t,\\\partial_1\nabla f_2&= \partial_1\nabla Y \nabla^3 Y_t +\partial_1\nabla^2 Y \nabla^2 Y_t +\nabla^2 Y \partial_1\nabla^2 Y_t +\nabla Y\partial_1\nabla^3Y_t.
\end{align*}
By the  H\"older and Sobolev inequalities and \eqref{es9}, we have
\begin{align}\label{esp1f2}
	\begin{split}
		\|\partial_1 f_2\|_{L^2}	&\lesssim  \|\partial_1\nabla Y\|_{L^\infty}\|\nabla^2Y_t\|_{L^2}+\|\nabla Y\|_{L^\infty}\|\partial_1\nabla^2Y_t\|_{L^2}\\&\lesssim \|\partial_1\nabla Y\|_{H^2}\|\Delta Y_t\|_{H^2}+\|\Delta Y\|_{H^2}\|\partial_1\Delta Y_t\|_{H^1},
	\end{split}
\end{align}
and
\begin{align}\label{esp1naf2}
	\begin{split}
		\|\partial_1\nabla  f_2\|_{L^2}	&\lesssim  \|\partial_1\nabla Y\|_{L^\infty}\|\nabla^3Y_t\|_{L^2}+\|\partial_1\nabla^2Y\|_{L^2}\|\nabla^2 Y_t\|_{L^\infty}\\&\quad+\|\nabla^2 Y\|_{L^\infty}\|\partial_1\nabla^2Y_t\|_{L^2}+\|\nabla Y\|_{L^\infty}\|\partial_1\nabla^3Y_t\|_{L^2}\\&\lesssim \|\partial_1\nabla Y\|_{H^2}\|\Delta Y_t\|_{H^2}+\|\Delta Y\|_{H^2}\|\partial_1\Delta Y_t\|_{H^1}.
	\end{split}
\end{align}
Thus, the combination of \eqref{esp1f2} and \eqref{esp1naf2} yields
\begin{align*}
	\begin{split}
		\|\partial_1 f_2\|_{H^1}&\lesssim (t+1)^{-1} \big(\sqrt{t+1}\|\partial_1\nabla Y\|_{H^2}\big)\big(\sqrt{t+1}\|\Delta Y_t\|_{H^2}\big)\\&\quad+(t+1)^{-1}\|\Delta Y\|_{H^2}\big((t+1)\|\p_1\Delta Y_t\|_{H^2}\big)\\&\lesssim (t+1)^{-1}E(t)^\frac12D(t)^\frac12.
	\end{split}
\end{align*}
\end{proof}

Now we turn to the proof of Lemma \ref{Lem2}.

\noindent
{\bf{ Estimate of $I_1=(f_1| Y_t-\frac{1}{4}\Delta Y-\f{1}{4}(t+1) \D Y_t)_{H^2}.$}}

By the Cauchy-Schwarz inequality and \eqref{esfH2}, we obtain
\begin{align*}
	|I_1|
&\lesssim \|f_1\|_{H^2}\big(\|Y_t\|_{H^2}+\|\Delta Y\|_{H^2}+(t+1)\|\Delta Y_t\|_{H^2}\big)\\
%
&\lesssim (t+1)^{-\frac{7}{12}}E(t)^\frac12D(t)^\frac12\left(E(t)^\frac12+(t+1)^{\frac12}D(t)^\frac12\right)\\&\lesssim (t+1)^{-\frac{7}{12}}E(t)D(t)^\frac12+E(t)^\frac12D(t).
\end{align*}
{\bf{Estimate of $I_2=(f_1|\frac{1}{16}(t+1)\partial_1^2\Delta Y)_{H^1}.$}}

By using the definition of $I_2$ and the integration by parts, we have
\begin{align*}
I_2
&=(f_1|\frac{1}{16}(t+1)\partial_1^2\Delta Y)_{L^2}-(\D f_1|\frac{1}{16}(t+1) \partial_1^2\Delta Y)_{L^2}.
\end{align*}
Applying the Cauchy-Schwarz inequality and using \eqref{esfH2}, we obtain
\begin{align*}
	|I_{2}|
&\lesssim \|f_1\|_{H^2}\|(t+1)\partial^2_1\Delta Y\|_{L^2}\\&\lesssim (t+1)^{-\frac{1}{12}}E(t)^\frac12D(t)^\frac12\sqrt{t+1}\|\partial^2_1\nabla Y\|_{H^1}\\
&\lesssim E(t)^{\frac12}D(t).
\end{align*}
{\bf{Estimate of $I_3=(f_1|\frac{1}{32}(t+1)^2\partial_1^2\Delta Y_t)_{H^1}$}}.

By using the integration by parts, we have		
\begin{align*}
	I_3&=-\big(\partial_1f_1|(t+1)^2\frac{1}{32}\partial_1\Delta Y_t\big)_{L^2}-\big(\partial_1\nabla f_1|\frac{1}{32}(t+1)^2\partial_1\nabla\Delta Y_t\big)_{L^2}.
\end{align*}
Hence, by\eqref{esp1f1H1}, one has
\begin{align*}
|I_3|&\lesssim \|\p_1f_1\|_{H^1}(t+1)^2\|\partial_1\Delta Y_t\|_{H^1}\\
%
&\lesssim (t+1)^{-1}E(t)^\frac12D(t)^\frac12 (t+1)D(t)^\frac12\lesssim E(t)^\frac12D(t).
\end{align*}

\noindent
{\bf{Estimate  of $J_1=(f_2| Y_t-\frac{1}{4}\Delta Y-\f{1}{4}(t+1) \D Y_t)_{H^2}$}}.

Similar to the estimate of $I_1$, by \eqref{esf2H2}, we have
\begin{align*}
|J_{1}|&\lesssim \|f_2\|_{H^2}\big(\|Y_t\|_{H^2}+\|
\Delta Y\|_{H^2}+ (t+1)\|\Delta Y_t\|_{H^2}\big)\\
%
&\lesssim (t+1)^{-\frac{7}{12}}E(t)^\frac12D(t)^\frac12\big(E(t)^\frac12+(t+1)^\frac12D(t)^\frac12\big)\\&\lesssim (t+1)^{-\frac{7}{12}}E(t)D(t)^\frac12+E(t)^\frac12D(t).
\end{align*}
{\bf{Estimate of  $J_2=\left(f_2|\frac{1}{16} (t+1)\partial_1^2\Delta Y\right)_{H^1}.$}}

Similar to the estimate of $I_2$, using integration by parts, together with \eqref{esf2H2}, we obtain
\begin{align*}
	J_2&\lesssim \|f_2\|_{H^2}\big((t+1)\|\partial^2_1\Delta Y\|_{L^2}\big)\\&\lesssim (t+1)^{-\frac{7}{12}}E(t)^\frac12D(t)^\frac12(t+1)^\frac12 \big(\sqrt{t+1}\|\partial^2_1\nabla Y\|_{H^1}\big)\\&\lesssim E(t)^\frac12 D(t).
\end{align*}
{\bf{Estimate of $J_3=\big(f_2|\frac{1}{32}(t+1)^2\partial_1^2\Delta Y_t\big)_{H^1}$}}.

Applying integration by parts, we get
\begin{align*}
	J_3&=\big(\partial_1f_2|-\frac{1}{32}(t+1)^2\partial_1\Delta Y_t\big)_{L^2}+\big(\partial_1\nabla f_2|-\frac{1}{32}(t+1)^2\partial_1\nabla\Delta Y_t\big)_{L^2}.
\end{align*}
By \eqref{esp1f2H1}, we get
\begin{align*}
	|J_{3}|&\lesssim  \|\partial_1f_2\|_{H^1}\|(t+1)^2\partial_1\Delta Y_t\|_{H^1}\lesssim E(t)^\frac12 D(t).
\end{align*}

In conclusion, combining the estimates of $I_1$ to $I_3$ and $J_1$ to $J_3$,
then performing time integration over $[0,t]$, we obtain
\begin{align}\label{f1f2}
	\begin{split}
			&\int_0^t \Big|\big( \nabla (\nabla Y\nabla Y_\tau) | Y_\tau -\frac14 \D Y -\f{1}{4}(\tau+1) \D Y_\tau\big)_{H^2}\Big|\,\mathrm{d}\tau\\
&\quad
		+\int_0^t\Big|\big(\nabla (\nabla Y\nabla Y_\tau) | \f{1}{16} (\tau+1)\D\p_1^2 Y+\f{1}{32} (\tau+1)^2 \D\p_1^2 Y_\tau \big)_{H^1}\Big| \,\mathrm{d}\tau\\
&\lesssim \int_0^t(\tau+1)^{-\frac{7}{12}}E(\tau)D(\tau)^\frac12\,\,\mathrm{d}\tau+\int_0^tE(\tau)^\frac12D(\tau)\,\mathrm{d}\tau\\
&\lesssim \sup_{0\leq \tau\leq t}E(\tau)\left(\int_0^tD(\tau)\,\mathrm{d}\tau\right)^\frac12\left(\int_{0}^{\infty}(\tau+1)^{-\frac{7}{6}}\,\mathrm{d}\tau\right)^\frac12+\left(\sup_{0\leq \tau\leq t}E(\tau)\right)^\frac12\int_0^tD(\tau)\,\mathrm{d}\tau\\&\lesssim \mathcal{E}(t)^\frac32.
	\end{split}
\end{align}

\subsection{Estimates of higher order nonlinear terms}
In this subsection, we deal with the higher order terms in $f$, that is \[f_3:=\nabla\Big(\big((\nabla Y)^2+(\nabla Y)^3+(\nabla Y)^4\big)\nabla Y_t\Big) \]
in \eqref{exp-f}.
Similar to \eqref{f12} in Section \ref{quad}, we have
\begin{align*}
&(f_3| Y_t-\frac{1}{4}\Delta Y-\f{1}{4}(t+1) \D Y_t)_{H^2}
+(f_3|\frac{1}{16}(t+1)\partial_1^2\Delta Y+\frac{1}{32}(t+1)^2\partial_1^2\Delta Y_t)_{H^1},\nonumber\\\nonumber
&\lesssim  \|f_3\|_{H^2}\big(\|Y_t\|_{H^2}+\|\Delta Y\|_{H^2}+(t+1)\|\Delta Y_t\|_{H^2}+\|(t+1)\partial^2_1\Delta Y\|_{L^2}\big)\\
&\quad +  \|\partial_1 f_3\|_{H^1}\|(t+1)^2\partial_1\Delta Y_t\|_{H^1} .
\end{align*}
Hence it reduces to the estimate $\|f_3\|_{H^2}$ and $\|\p_1f_3\|_{H^1}$. For $f_3$, we write
\begin{align*}
	f_3=\sum_{k=2}^{4}\big[k(\nabla Y)^{k-1}\nabla^2Y\nabla Y_t+(\nabla Y)^k\nabla^2 Y_t\big]=\sum_{k=2}^{4}\big[k(\nabla Y)^{k-1}f_1+(\nabla Y)^{k-1}f_2\big] ,
\end{align*}
where $f_1$ and $f_2$ are defined by \eqref{f12-d} in Section \ref{quad}. By  the H\"older and Sobolev inequalities, we get for $k=2$ and $i=1,2$,
\begin{align*}
	\|\nabla Yf_i\|_{H^2}\lesssim \|\nabla Yf_i\|_{L^2}+\|\nabla^2(\nabla Yf_i)\|_{L^2}\lesssim \|\Delta Y\|_{H^2}\|f_i\|_{H^2}.
\end{align*}
For $k=3$ and $i=1,2$, we have
\begin{align*}
	\|(\nabla Y)^2f_i\|_{H^2}&\lesssim \|\nabla Y(\nabla Yf_i)\|_{L^2}+\|\nabla^2\big(\nabla Y(\nabla Yf_i)\big)\|_{L^2}\\&\lesssim \|\Delta Y\|_{H^2}\|\nabla Yf_i\|_{H^2}\lesssim \|\Delta Y\|_{H^2}^2\|f_i\|_{H^2}.
\end{align*}
Similarly, for $k=4$ and $i=1,2$,
\begin{align*}
	\|(\nabla Y)^3f_i\|_{H^2}\lesssim \|\Delta Y\|_{H^2}^3\|f_i\|_{H^2}.
\end{align*}
Summing up the above estimates and noting that $\mathcal{E}(t)\leq \delta$, we get
\begin{align*}
	\|f_3\|_{H^2}&\lesssim (1+\|\Delta Y\|_{H^2}^2)\|\Delta Y\|_{H^2}(\|f_1\|_{H^2}+\|f_2\|_{H^2})\\&\lesssim (1+\delta)\|\Delta Y\|_{H^2}(\|f_1\|_{H^2}+\|f_2\|_{H^2})\\&\lesssim\|\Delta Y\|_{H^2}(\|f_1\|_{H^2}+\|f_2\|_{H^2}).
\end{align*}
Hence, it follows from \eqref{esfH2} and \eqref{esf2H2} that
\begin{align}\label{esf3H2}
	\begin{split}
		\|f_3\|_{H^2}&\lesssim \|\Delta Y\|_{H^2}(t+1)^{-\frac{7}{12}}E(t)^\frac12D(t)^\frac12\\&\lesssim (t+1)^{-\frac{7}{12}}E(t)D(t)^\frac12.
	\end{split}
\end{align}

On the other hand, note that
\begin{align*}
	\p_1f_3=\sum_{k=2}^{4}\big[k\p_1\big((\nabla Y)^{k-1}f_1\big)+\p_1\big((\nabla Y)^{k-1}f_2\big)\big].
\end{align*}
Then for $k=2$ and $i=1,2$, one has
\begin{align*}
	\|\partial_1(\nabla Yf_i)\|_{H^1}&\lesssim \|\partial_1(\nabla Yf_i)\|_{L^2}+\|\nabla \partial_1(\nabla Yf_i)\|_{L^2}\\&\lesssim \|\Delta Y\|_{H^2}\|\p_1f_i\|_{H^1}+\|\p_1\nabla  Y\|_{H^2}\|f_i\|_{H^2}.
\end{align*}
For $k=3$ and $i=1,2$,
\begin{align*}
	\|\p_1\big((\nabla Y)^{2}f_i\big)\|_{H^1}&\lesssim \|\partial_1\big(\nabla Y(\nabla Y f_i)\big)\|_{L^2}+\|\nabla \partial_1\big(\nabla Y(\nabla Yf_i)\big)\|_{L^2}\\&\lesssim \|\Delta Y\|_{H^2}\|\partial_1(\nabla Yf_i)\|_{H^1}+\|\p_1\nabla  Y\|_{H^2}\|\nabla Yf_i\|_{H^2}\\&\lesssim \|\Delta Y\|_{H^2}^2\|\p_1f_i\|_{H^1}+\|\Delta Y\|_{H^2}\|\p_1\nabla  Y\|_{H^2}\|f_i\|_{H^2}.
\end{align*}
Similarly, for $k=4$ and $i=1,2$,
\begin{align*}
	\|\p_1\big((\nabla Y)^{3}f_i\big)\|_{H^1}\lesssim \|\Delta Y\|_{H^2}^3\|\p_1f_i\|_{H^1}+\|\Delta Y\|_{H^2}^2\|\p_1\nabla  Y\|_{H^2}\|f_i\|_{H^2}.
\end{align*}
Hence, gathering the above estimates and using the ansatz that $\mathcal{E}(t)\leq \delta$, we derive
\begin{align*}
	\|\p_1f_3\|_{H^1}\lesssim \sum_{i=1}^{2}\big[\|\Delta Y\|_{H^2}\|\p_1f_i\|_{H^1}+\|\p_1\nabla  Y\|_{H^2}\|f_i\|_{H^2}\big].
\end{align*}
Consequently, by \eqref{esfH2}, \eqref{esf2H2}, \eqref{esp1f1H1} and \eqref{esp1f2H1}, one has
\begin{align}\label{esp1f3H1}
	\begin{split}
		\|\p_1f_3\|_{H^1}&\lesssim \|\Delta Y\|_{H^2}(t+1)^{-1}E(t)^\frac12D(t)^\frac12\\&\quad+(t+1)^{-\frac12}\big(\sqrt{t+1}\|\p_1\nabla  Y\|_{H^2}\big)(t+1)^{-\frac{7}{12}}E(t)^\frac12D(t)^\frac12\\&\lesssim (t+1)^{-1}E(t)D(t)^\frac12+(t+1)^{-\frac{13}{12}}E(t)D(t)^\frac12.
	\end{split}
\end{align}
Now we take the integral in time over $[0,t]$,
by \eqref{esf3H2} and \eqref{esp1f3H1},
using the similar method as that in Section \ref{quad}, we derive
\begin{align}\label{f3}
	\begin{split}
		&\int_0^t \Big|\big( f_3| Y_\tau -\frac14 \D Y -\f{1}{4}(\tau+1) \D Y_\tau\big)_{H^2}\big|\,\mathrm{d}\tau\\&\quad
		+\int_0^t\Big|\big(f_3| \f{1}{16} (\tau+1)\D\p_1^2 Y+\f{1}{32} (\tau+1)^2 \D\p_1^2 Y_\tau \big)_{H^1}\big| \,\mathrm{d}\tau\\&\lesssim \mathcal{E}(t)^2.
	\end{split}
\end{align}

\subsection{Estimate of the pressure term}
In this subsection, we are going to estimate the pressure term in $f$, that is $A\nabla p$ in \eqref{exp-f} with
\begin{align}\label{eqnap}
	\begin{split}
		\nabla p&=-\Delta^{-1}\nabla \textrm{div}\big(( A^\top A-I) \nabla p\big)\\&\quad+\Delta^{-1}\nabla \operatorname{div}\Bigl(A^\top\operatorname{div}\big(A^\top(\p_1 Y\otimes \p_1Y-Y_t\otimes  Y_t)\big)\Bigr).
	\end{split}
\end{align}
From the estimates in the above two subsections, we need to estimate $\|A \nabla p\|_{H^2}$ and $\|\p_1(A\nabla p)\|_{H^1}$.

We first derive some estimates about $A$.
Using \eqref{expA} and the ansatz that $\mathcal{E}(t)\leq \delta$, we get
\begin{align}\label{esALinf}
	\|A\|_{L^\infty}&\lesssim 1+\|\nabla Y\|_{L^\infty}+\|\nabla Y\|_{L^\infty}^2\lesssim 1+\|\Delta Y\|_{H^1}+\|\Delta Y\|_{H^1}^2\lesssim  1,
\end{align}
and
\begin{align}\label{esnaAH2}
	\|\nabla A\|_{H^2}&\lesssim \|\nabla^2 Y\|_{H^2}+\|\nabla^2 Y\|_{H^2}^2\lesssim \|\Delta Y\|_{H^2}.
\end{align}
Moreover, due to \eqref{exp-A2}, we deduce
\begin{equation}\label{esnaA2}
	\|\nabla(A^\top A-I)\|_{H^1}\lesssim \|\Delta Y\|_{H^2}.
\end{equation}

Now we are ready to derive the estimate of $\|A\nabla p\|_{H^2}$.
Note that in \eqref{eqnap}, $\mathcal{R}:=\Delta^{-1}\nabla \textrm{div}$ is a Riesz transform.
By the boundedness of $\mathcal{R}$ in $L^2$, together with the H\"older and Sobolev inequalities, we have
\begin{align}\label{nap1}
	\|\mathcal{R}\big(( A^\top A-I) \nabla p\big)\|_{H^2}\lesssim \|\nabla( A^\top A-I)\|_{H^1}\|\nabla p\|_{H^2} \lesssim \|\Delta Y\|_{H^2}\|\nabla p\|_{H^2}.
\end{align}
On the other hand, using \eqref{esALinf} and \eqref{esnaAH2}, we obtain
\begin{align}\label{nap2-0}
	\begin{split}
		&\|\mathcal{R}\Bigl(A^\top\operatorname{div}\big(A^\top(\p_1 Y\otimes \p_1Y-Y_t\otimes  Y_t)\big)\Bigr)\|_{H^2}\\&\lesssim (\|A\|_{L^\infty}+\|\nabla A\|_{H^1}) \|\nabla \big(A^\top(\p_1 Y\otimes \p_1Y-Y_t\otimes  Y_t)\big)\|_{H^2}\\&\lesssim\|\nabla \big(A^\top(\p_1 Y\otimes \p_1Y-Y_t\otimes  Y_t)\big)\|_{H^2}.
	\end{split}
\end{align}
Then by \eqref{esALinf} and \eqref{esnaAH2}, we compute
\begin{align}\label{nap2-1}
	\begin{split}
		&\|\nabla \big(A^\top(\p_1 Y\otimes \p_1Y-Y_t\otimes  Y_t)\big)\|_{H^2} \\&\lesssim \|\nabla A\|_{H^2}\big(\|\nabla \p_1Y\otimes \p_1Y\|_{H^1}+\|\nabla Y_t\otimes  Y_t\|_{H^1}\big) \\&\quad+(\|A\|_{L^\infty}+\|\nabla A\|_{H^1})\big(\|\nabla\p_1 Y\otimes \p_1Y\|_{H^2}+\|\nabla Y_t\otimes  Y_t\|_{H^2}\big) \\&\lesssim \|\nabla \p_1 Y\|_{H^2}^2+\|\nabla Y_t\|_{H^2}^2.
	\end{split}
\end{align}
Combining \eqref{nap2-0} with \eqref{nap2-1}, one has
\begin{align}\label{nap2}
	\begin{split}
		&\Vert\mathcal{R}\Bigl(A^\top\operatorname{div}\big(A^\top(\p_1 Y\otimes \p_1Y-Y_t\otimes  Y_t)\big)\Bigr)\Vert_{H^2}\lesssim \|\nabla \p_1 Y\|_{H^2}^2+\|\nabla Y_t\|_{H^2}^2.
	\end{split}
\end{align}
Then it follows from \eqref{eqnap}, \eqref{nap1} and \eqref{nap2} that
\begin{align}\label{napH2-1}
	\|\nabla p\|_{H^2}&\leq C\|\Delta Y\|_{H^2}\|\nabla p\|_{H^2}+C (\|\nabla \partial_1Y\|_{H^2}^2+\|\nabla Y_t\|_{H^2}^2).
\end{align}
By taking $\delta$ small enough such that \[C\mathcal{E}(t)\leq C\delta \leq \frac12,\]
then we get from \eqref{napH2-1} that
\begin{align}\label{napH2}	\|\nabla p\|_{H^2}&\lesssim \|\nabla \partial_1Y\|_{H^2}^2+\|\nabla Y_t\|_{H^2}^2.
\end{align}
Moreover, by \eqref{esALinf}, \eqref{esnaAH2} and \eqref{napH2}, we have
\begin{align}\label{Anap}
	\|A\nabla p\|_{H^2}\lesssim (1+\|A\|_{L^\infty}+\|\nabla A\|_{H^1})\|\nabla p\|_{H^2}\lesssim \|\nabla \partial_1Y\|_{H^2}^2+\|\nabla Y_t\|_{H^2}^2.
\end{align}

Now we consider the term
\begin{align*}
	K&=	 ( A\na p | Y_t -\frac14 \D Y -\f{1}{4} (t+1) \D Y_t)_{H^2}
	\\&\quad+( A\na p |\f{1}{16} (t+1) \p_1^2\D Y)_{H^1}+(A \na p | -\f{1}{16} (t+1)^2\D\p_1^2 Y_t )_{H^1}\\&:=K_{11}+K_{12}+K_{13}.
\end{align*}
For $K_{11}$, using the Cauchy-Schwarz inequality and \eqref{Anap}, we obtain
\begin{align*}
	|K_{11}|&\lesssim \|A\nabla p\|_{H^2}\big(\|Y_t\|_{H^2}+\|\Delta Y\|_{H^2}+\|(t+1)\Delta Y_t\|_{H^2}\big)\\&\lesssim \big(\|\nabla \partial_1Y\|_{H^2}^2+\|\nabla Y_t\|_{H^2}^2\big)\big(\|Y_t\|_{H^2}+\|\Delta Y\|_{H^2}\big)\\&\quad+(\sqrt{t+1}\|\nabla \partial_1Y\|_{H^2})\|\nabla \partial_1Y\|_{H^2}\big(\sqrt{t+1}\|\Delta Y_t\|_{H^2}\big)\\&\quad+(\sqrt{t+1}\|\nabla Y_t\|_{H^2})\|\nabla Y_t\|_{H^2}\big(\sqrt{t+1}\|\Delta Y_t\|_{H^2}\big)\\&\lesssim E(t)^\frac12 D(t).
\end{align*}
To deal with $K_{12}$, applying integration by parts, we get
\begin{align*}
	K_{12}=( A\na p |\f{1}{16} (t+1) \p_1^2\D Y)_{L^2}+(\Delta( A\na p) |-\f{1}{16} (t+1) \p_1^2\D Y)_{L^2},
\end{align*}
which gives rise to
\begin{align*}
	|K_{12}|&\lesssim \|A\nabla p\|_{H^2}\|(t+1)\partial_1^2\Delta Y\|_{L^2}\\&\lesssim (\sqrt{t+1}\|\nabla \partial_1Y\|_{H^2})\|\nabla \partial_1Y\|_{H^2}\big(\sqrt{t+1}\|\partial_1^2\nabla Y\|_{H^1}\big)\\&\quad+(\sqrt{t+1}\|\nabla Y_t\|_{H^2})\|\nabla Y_t\|_{H^2}\big(\sqrt{t+1}\|\partial_1^2\nabla Y\|_{H^1}\big)\\&\lesssim E(t)^\frac12 D(t).
\end{align*}
For  $K_{13}$, applying integration by parts, we get
\begin{align}\label{K13}
	\begin{split}
		|K_{13}|&=|(\partial_1(A \na p) | \f{1}{16} (t+1)^2\D\p_1 Y_t )_{H^1}|\\&\lesssim \|\partial_1(A \na p)\|_{H^1}(t+1)^2\|\D\p_1 Y_t\|_{H^1}.
	\end{split}
\end{align}

It remains to estimate $\|\partial_1(A \na p)\|_{H^1}$. By \eqref{expA}, we get
\begin{align}\label{esp1naA}
	\begin{split}
		\|\partial_1A\|_{H^2}&\lesssim \|\nabla\partial_1 Y\|_{H^2}+\|\nabla\partial_1 Y\|_{H^2}\|\D Y\|_{H^2}\\&\lesssim \|\nabla\partial_1 Y\|_{H^2}(1+\|\Delta Y\|_{H^2})\lesssim \|\nabla\partial_1 Y\|_{H^2}.
	\end{split}
\end{align}
In addition, due to \eqref{exp-A2}, we get
\begin{align}\label{esp1A2}
	\|\p_1( A^\top A-I)\|_{H^1}\lesssim \|\nabla\partial_1 Y\|_{H^1}.
\end{align}
Note that
\begin{align*}
	\|\partial_1(A \na p)\|_{H^1}\lesssim \|\partial_1A \na p\|_{H^1}+\|A \p_1\na p\|_{H^1}.
\end{align*}
It follows from \eqref{napH2} and \eqref{esp1naA} that
\begin{align}\label{p1Anp}
	\|\partial_1A \na p\|_{H^1}\lesssim \|\partial_1A\|_{H^2}\|\nabla p\|_{H^1}\lesssim \|\nabla\partial_1 Y\|_{H^2}\big(\|\nabla \partial_1Y\|_{H^2}^2+\|\nabla Y_t\|_{H^2}^2\big).
\end{align}
To deal with $\|A \p_1\na p\|_{H^1}$, we first estimate $\|\p_1\na p\|_{H^1}$. From \eqref{eqnap}, \eqref{esnaA2} and \eqref{esp1A2}, we could derive that
\begin{align}\label{p1nap1}
	\begin{split}
		&\|\mathcal{R}\p_1\big(( A^\top A-I) \nabla p\big)\|_{H^1}\\&\lesssim \|\p_1( A^\top A-I)\|_{H^1}\|\nabla p\|_{H^2}+\|\nabla( A^\top A-I)\|_{H^1}\|\p_1\nabla p\|_{H^1} \\&\lesssim \|\nabla \p_1Y\|_{H^2}\big(\|\nabla \partial_1Y\|_{H^2}^2+\|\nabla Y_t\|_{H^2}^2\big)+\|\Delta Y\|_{H^2}\|\p_1\nabla p\|_{H^1}.
	\end{split}
\end{align}
On the other hand, using \eqref{esALinf}, \eqref{esnaAH2}, \eqref{nap2-1} and \eqref{esp1naA}, we infer
\begin{align}\label{p1nap2-0}
	\begin{split}
		&\|\mathcal{R}\p_1\Bigl(A^\top\operatorname{div}\big(A^\top(\p_1 Y\otimes \p_1Y-Y_t\otimes  Y_t)\big)\Bigr)\|_{H^1}\\&\lesssim \|\p_1 A\|_{H^2}\|\nabla \big(A^\top(\p_1 Y\otimes \p_1Y-Y_t\otimes  Y_t)\big)\|_{H^1}\\&\quad+(\|A\|_{L^\infty}+\|\nabla A\|_{H^1})\|\nabla\p_1 \big(A^\top(\p_1 Y\otimes \p_1Y-Y_t\otimes  Y_t)\big)\|_{H^1}
		\\&\lesssim \|\nabla \p_1 Y\|_{H^2}\big(\|\nabla \p_1 Y\|_{H^2}^2+\|\nabla Y_t\|_{H^2}^2\big)\\&\quad+ \|\nabla\p_1 \big(A^\top(\p_1 Y\otimes \p_1Y-Y_t\otimes  Y_t)\big)\|_{H^1}.
	\end{split}
\end{align}
Then applying \eqref{esALinf}, \eqref{esnaAH2} and \eqref{esp1naA}, we compute
\begin{align}\label{p1nap2-1}
	\begin{split}
		&\|\nabla\p_1 \big(A^\top(\p_1 Y\otimes \p_1Y-Y_t\otimes  Y_t)\big)\|_{H^1}\\&\lesssim \|\p_1 A\|_{H^2}\big(\|\nabla\p_1 Y\otimes \p_1Y\|_{H^2}+\|\nabla Y_t\otimes  Y_t\|_{H^2}\big)\\&\quad+(\|A\|_{L^\infty}+\|\nabla A\|_{H^1})\big(\|\nabla\p_1(\p_1 Y\otimes \p_1Y)\|_{H^1}+\|\nabla\p_1( Y_t\otimes  Y_t)\|_{H^1}\big)\\&\lesssim \|\nabla\partial_1 Y\|_{H^2}\big(\|\nabla \p_1 Y\|_{H^2}^2+\|\nabla Y_t\|_{H^2}^2\big)\\&\quad+\|\nabla \p_1 Y\|_{H^2}\|\nabla \p_1^2 Y\|_{H^1}+\|\nabla Y_t\|_{H^2}\|\nabla\p_1 Y_t\|_{H^1}.
	\end{split}
\end{align}
The combination of \eqref{p1nap2-0} and \eqref{p1nap2-1} yields
\begin{align}\label{p1nap2}
	\begin{split}
		&\|\mathcal{R}\p_1\Bigl(A^\top\operatorname{div}\big(A^\top(\p_1 Y\otimes \p_1Y-Y_t\otimes  Y_t)\big)\Bigr)\|_{H^1}\\&\lesssim \|\nabla\partial_1 Y\|_{H^2}\big(\|\nabla \p_1 Y\|_{H^2}^2+\|\nabla Y_t\|_{H^2}^2\big)\\&\quad+\|\nabla \p_1 Y\|_{H^2}\|\nabla \p_1^2 Y\|_{H^1}+\|\nabla Y_t\|_{H^2}\|\nabla\p_1 Y_t\|_{H^1}.
	\end{split}
\end{align}
Similar to the derivation of \eqref{napH2}, according to \eqref{p1nap1} and \eqref{p1nap2}, if we take $\delta$ sufficiently small, we can derive
\begin{align*}
	\|\p_1\nabla p\|_{H^1}&\lesssim \|\nabla \p_1Y\|_{H^2}\big(\|\nabla \partial_1Y\|_{H^2}^2+\|\nabla Y_t\|_{H^2}^2\big)\\&\quad+\|\nabla \p_1 Y\|_{H^2}\|\nabla \p_1^2 Y\|_{H^1}+\|\nabla Y_t\|_{H^2}\|\nabla\p_1 Y_t\|_{H^1}.
\end{align*}
Hence,
\begin{align}\label{Ap1np}
	\begin{split}
		\|A\p_1\nabla p\|_{H^1}&\lesssim (\|A\|_{L^\infty}+\|\nabla A\|_{H^1})\|\p_1\nabla p\|_{H^1}\\&\lesssim \|\nabla \p_1Y\|_{H^2}\big(\|\nabla \partial_1Y\|_{H^2}^2+\|\nabla Y_t\|_{H^2}^2\big)\\&\quad+\|\nabla \p_1 Y\|_{H^2}\|\nabla \p_1^2 Y\|_{H^1}+\|\nabla Y_t\|_{H^2}\|\nabla\p_1 Y_t\|_{H^1}.
	\end{split}
\end{align}
Combining \eqref{p1Anp} and \eqref{Ap1np}, we deduce
\begin{align}\label{esp1Anap}
	\begin{split}
		\|\partial_1(A \na p)\|_{H^1}&\lesssim \|\nabla\partial_1 Y\|_{H^2}\big(\|\nabla \partial_1Y\|_{H^2}^2+\|\nabla Y_t\|_{H^2}^2\big)\\&\quad+\|\nabla \p_1 Y\|_{H^2}\|\nabla \p_1^2 Y\|_{H^1}+\|\nabla Y_t\|_{H^2}\|\nabla\p_1 Y_t\|_{H^1}.
	\end{split}
\end{align}
Plugging \eqref{esp1Anap} into \eqref{K13}, we obtain
\begin{align*}
	|K_{13}|&\lesssim \left[(t+1)\left(\|\nabla\partial_1 Y\|_{H^2}^2+\|\nabla Y_t\|_{H^2}^2\right)\right]\|\nabla\partial_1 Y\|_{H^2}(t+1)\|\D\p_1 Y_t\|_{H^1}\\&\quad+\big(\sqrt{t+1}\|\nabla \partial_1 Y\|_{H^2}\big)\big(\sqrt{t+1}\|\nabla\partial_1^2 Y\|_{H^1}\big)(t+1)\|\D\p_1 Y_t\|_{H^1}\\&\quad+\left((t+1)\|\nabla \p_1 Y_t\|_{H^1}\right)\|\nabla Y_t\|_{H^1}(t+1)\|\D\p_1 Y_t\|_{H^1}\\&\lesssim E(t)D(t)+E(t)^\frac12D(t).
\end{align*}
As a consequence,
\begin{align}\label{f4}
	\begin{split}
			&\int_0^t \Big|\big( A\nabla p| Y_\tau -\frac14 \D Y -\f{1}{4}(\tau+1) \D Y_\tau\big)_{H^2}\big|\,\mathrm{d}\tau\\&\quad
		+\int_0^t\Big|\big(A\nabla p| \f{1}{16} (\tau+1)\D\p_1^2 Y+\f{1}{32} (\tau+1)^2 \D\p_1^2 Y_\tau \big)_{H^1}\big| \,\mathrm{d}\tau\\&\lesssim \mathcal{E}(t)^2+\mathcal{E}(t)^\frac32.
	\end{split}
\end{align}

The combination of \eqref{f1f2}, \eqref{f3} and \eqref{f4} completes the proof of Lemma \ref{Lem2}.

\section{Proof of Theorem \ref{th1}}\label{secproof1}
In this section, we prove Theorem \ref{th1}.
\begin{proof}[Proof of Theorem \ref{th1}]
The local well-posedness of \eqref{A10} can be obtained by a standard argument. To extend the local solution to a global one, we only need to prove the global \textit{a priori} estimate.
 Indeed, according to Lemma \ref{Lem1} and Lemma \ref{Lem2}, we get that if $\mathcal{E}(t)\leq \delta$ where $\delta$ is the small constant determined in Lemma \ref{Lem2}, then
	\begin{align}\label{Et}
		\mathcal{E}(t)\leq C_0(\|  Y_1\|_{H^3}^2
		+\|\pa_1Y_0\|_{H^3}^2
		+\|  \D Y_0\|_{H^2}^2)+C_0\mathcal{E}(t)^\frac32+C_0\mathcal{E}(t)^2
	\end{align}
holds for some constant $C_0>1$. In the sequel, we will complete the proof by using the standard  continuity argument. Firstly, by the definition of $\mathcal{E}(t)$, we get for some constant $C_1>1$ that
\begin{equation*}
	\mathcal{E}(0)\leq C_1(\|  Y_1\|_{H^3}^2
	+\|\pa_1Y_0\|_{H^3}^2
	+\|  \D Y_0\|_{H^2}^2).
\end{equation*}
Let us take
\begin{equation*}
   M=\max\{2C_0, C_1\},\quad\epsilon_0=\frac{1}{2M}\min\{\delta,\frac{1}{16C_0^2}\},
\end{equation*}
then under the assumption \eqref{asmp}, we have
\begin{equation*}
	\mathcal{E}(0)\leq M\epsilon_0.
\end{equation*}
By the continuity of the energy, there holds
\begin{equation*}
	\mathcal{E}(t)\leq 2M\epsilon_0,
\end{equation*}
for a  fixed time interval depending only on the initial energy.
Then $\mathcal{E}(t)\leq \delta$, hence \eqref{Et} holds. Therefore, we deduce from \eqref{asmp} that
\begin{align*}
	\mathcal{E}(t)&\leq C_0\epsilon_0+C_0\left(\sqrt{\mathcal{E}(t)}+\mathcal{E}(t)\right)\mathcal{E}(t)\\&\leq \frac{M}{2}\epsilon_0+C_0\left(\sqrt{2M\epsilon_0}+2M\epsilon_0\right)\mathcal{E}(t)\\&\leq \frac{M}{2}\epsilon_0+\frac12\mathcal{E}(t).
\end{align*}
This implies that
\begin{align*}
	\mathcal{E}(t)\leq M\epsilon_0.
\end{align*}
By the continuity argument, there holds $\mathcal{E}(t)\leq M\epsilon_0$ for all $t>0$ under the assumption \eqref{asmp}.
\end{proof}

\section{Proof of Theorem \ref{MTH}}

Now we are in a position to complete the proof of Theorem
\ref{MTH}: the global well-posedness in Eulerian coordinates.  Let us first recall Proposition 6.1 from \cite{XZ}.

\linespread{1.2}
{\begin{prop}\label{p5.1}
Let  $b_0-e_1\in H^{s}(\R^3)$ and $u_0\in H^s(\R^3)$ for
$s\in(\f32,3]$, \eqref{A1} has a unique solution $(u, b)$ on $[0,T]$ for
some $T>0$ so that $b-e_1\in C([0,T]; H^{s}(\R^3)),$ $  u\in
C([0,T]; H^{s}(\R^3))$ with $\na u\in L^2((0,T);H^s(\R^3))$ and $
\na p\in C([0,T]; H^{s-1}(\R^3)).$ Moreover, if $T^\ast$ is the life
span to this solution, and $T^\ast<\infty,$ one has
\begin{equation*}
\int_0^{T^\ast}\bigl(\|\na u(t)\|_{L^\infty}+\|
b(t)\|_{L^\infty}^2\bigr)\,\mathrm{d}t=\infty.
\end{equation*}
\end{prop}}

  For the given initial
data $(u_0,b_0)$ which satisfies the assumptions of Theorem
\ref{MTH}, we deduce from Proposition \ref{p5.1} that \eqref{A1} has
a unique solution $(u,b)$ on $[0,T^\ast)$ such that for any
$T<T^\ast,$
\beno
b-e_1\in C([0,T]; H^{3}(\R^3)),
\quad  u\in
C([0,T]; H^{3}(\R^3))\quad\textrm{with}\quad
\na u\in L^2(0,T;H^3(\R^3)).
\eeno
Moreover, it follows from the transport equation of \eqref{A1} that
\beno \|b(t)\|_{L^\infty}\leq \|b_0\|_{L^\infty}\exp\left\{\|\na
u\|_{L^1_t(L^\infty)}\right\}. \eeno Therefore, by virtue of
Proposition \ref{p5.1}, in order to complete the existence part of
Theorem \ref{MTH}, it remains to prove that
\beq \label{e5.2}
\int_0^{T^\ast}\|\na u(t)\|_{L^\infty}\,\mathrm{d}t<\infty.
\eeq
  Indeed, by \eqref{es2},  it
follows from the Lagrangian formulation   that for any
$T<T^\ast$,
\begin{align*}
\int_0^{T}\|\na u(t)\|_{L^\infty}\,\mathrm{d}t&\leq
\int_0^{T}\bigl\|A\na_y Y_t(t)\|_{L^\infty}\,\mathrm{d}t \\
&\lesssim \bigl(1+\|\na Y\|_{L^\infty_{t,y}}\bigr)
\int_0^T \|\na_y Y_t(t)\|_{L^\infty}\,\mathrm{d}t\\
&\lesssim \int_0^T \|\nabla^2 Y_t\|_{H^1}^{\frac{5}{6}}\|\partial_1\nabla^2 Y_t\|_{H^1}^{\frac{1}{6}} \,\mathrm{d}t\\
&\lesssim \int_0^T (1+t)^{-\frac{7}{12}}
\big((1+t)^{\f12}\|\Delta Y_t\|_{H^2}\big)^{\frac{5}{6}}
\big({(1+t)}\|\partial_1\Delta Y_t\|_{H^1}\big)^{\frac{1}{6}}\,\mathrm{d}t\\
&\lesssim \mathcal{E}^{\f12} \lesssim \epsilon_0^{\f12}.
\end{align*}
This verifies \eqref{e5.2} and thus  Theorem \ref{MTH} is proved.

\section*{Acknowledgement}
We would like to thank Professor Fanghua Lin for profitable  suggestions and encouragements on this topic. 
Cai was supported in part by NSFC grants (no. 12201122)
and by the Initiative Funding for New Researchers, Fudan University. Han was partially supported by Zhejiang Province Science fund (no. LY21A010009). Zhao was  supported in part by  NSFC grants (no. 12301256), Shanghai Sailing Program (no.
22YF1412300), a funding from Laboratory  of Computational  Physics  (no. 6142A05QN21007) and the Shanghai University of Finance and Economics startup funding.

\end{document}